\documentclass[a4paper, 9pt, twosides]{amsart}

\usepackage{lmodern}
\usepackage{amssymb}
\usepackage{amsthm}

\usepackage{cite}
\usepackage{txfonts}
\usepackage{enumitem}
\usepackage{dsfont}
\usepackage{graphicx}
\usepackage[usenames, dvipsnames]{color}
\usepackage{tikz}
\usetikzlibrary{automata, positioning, shapes}
\usetikzlibrary{shapes,snakes,calc}
\usetikzlibrary{decorations.pathreplacing}
\usepackage{mathtools}
\allowdisplaybreaks
\usepackage{todonotes}
\usepackage{algorithm,algorithmic}
\usepackage{multicol}
\usepackage{caption}
\usepackage{subcaption}

\newtheorem{theorem}{Theorem}
\newtheorem{proposition}{Proposition}
\newtheorem{lemma}{Lemma}

\newtheorem{defn}{Definition}
\newtheorem{example}{Example}

\newtheorem{prob}{Problem}

\newtheorem{rem}{Remark}

\newtheorem{assump}{Assumption}

\newtheorem{experiment}{Experiment}

\newcommand{\abs}[1]{\left\lvert{#1}\right\rvert}

\newcommand{\norm}[1]{\left\lVert#1\right\rVert}
\newcommand{\pmat}[1]{\begin{pmatrix}#1\end{pmatrix}}

\newcommand{\R}{\mathbb{R}}
\newcommand{\N}{\mathbb{N}}
\newcommand{\Svec}{\mathcal{S}}
\newcommand{\PP}{\mathbb{P}}
\newcommand{\EE}{\mathbb{E}}
\newcommand{\D}{\mathcal{D}}
\renewcommand{\P}{\mathcal{P}}
\allowdisplaybreaks

\newcommand{\is}{i_s}
\newcommand{\iu}{i_u}

\DeclareMathOperator{\minimize}{minimize}
\DeclareMathOperator{\sbjto}{subject\;to}

\newcommand{\diag}{\text{diag}}

\title[]{Design of periodic scheduling and control for\\networked systems under random data loss}

\author[A.\ Kundu]{Atreyee Kundu}
\author[D.\ E.\ Quevedo]{Daniel E. Quevedo}
\thanks{Atreyee Kundu is with the Department of Electrical Engineering, Indian Institute of Science Bangalore, India. Daniel E.\ Quevedo is with the School of Electrical Engineering and Robotics, Queensland University of Technology, Australia. Emails: \texttt{atreyeek@iisc.ac.in, dquevedo@ieee.org}}

\date{\today}

\begin{document}

    \begin{abstract}
    	{This paper deals with Networked Control Systems (NCSs) whose shared networks have limited communication capacity and are prone to data losses. We assume that among \(N\) plants, only \(M\:(< N)\) plants can communicate with their controllers at any time instant. In addition, a control input, at any time instant, is lost in a channel with a probability \(p\). Our contributions are threefold. First, we identify necessary and sufficient conditions on the open-loop and closed-loop dynamics of the plants that ensure existence of purely time-dependent periodic scheduling sequences under which stability of each plant is preserved for all admissible data loss signals. Second, given the open-loop and closed-loop dynamics of the plants, relevant parameters of the shared network and a period for the scheduling sequence, we present an algorithm that verifies our stability conditions and if satisfied, designs stabilizing scheduling sequences. Otherwise, the algorithm reports non-existence of a stabilizing periodic scheduling sequence with the given period and stability margins. Third, given the plant matrices, the parameters of the network and a period for the scheduling sequence, we present an algorithm that designs static state-feedback controllers such that our stability conditions are satisfied. The main apparatus for our analysis is a switched systems representation of the individual plants in an NCS whose switching signals are time-inhomogeneous Markov chains. Our stability conditions rely on the existence of sets of symmetric and positive definite matrices that satisfy certain (in)equalities.}
    \end{abstract}
\maketitle
\section{Introduction}
\label{s:intro}
\subsection{Problem setting}
\label{ss:prob_setting}
	Networked Control Systems (NCSs) are spatially distributed control systems in which the communication between plants and their controllers occurs through shared networks. NCSs find wide applications in sensor networks, remote surgery, haptics collaboration over the internet, automated highway systems, unmanned aerial vehicles, etc. \cite{Hespanha2007}. While the use of shared communication networks in NCSs offers flexible architectures and reduced installation and maintenance costs, the exchange of information between the plants and their controllers often suffers from network induced limitations and uncertainties.

    In this paper we deal with NCSs whose communication networks have limited bandwidth and are prone to data losses. Examples of communication networks with limited bandwidth include wireless networks (an important component of smart home, smart transportation, smart city, remote surgery, platoons of autonomous vehicles, etc.) and underwater acoustic communication systems \cite{quevedo2020}. The scenario in which the number of plants sharing a communication network is higher than the capacity of the network is called \emph{medium access constraint}. This scenario motivates a need to allocate the communication network to each plant in a manner so that good qualitative properties of the plants are preserved. This task of efficient allocation of a shared communication network is commonly referred to as a \emph{scheduling problem}, and the corresponding allocation scheme is called a \emph{scheduling {sequence}}. Under ideal communication, both a plant and its controller receive the intended information whenever the shared network is allocated to them. However, in practical situations, communication networks are often prone to uncertainties like intermittent data losses. In particular, data loss is a common feature for noisy communication networks. For instance, in cloud-aided vehicle control systems, the control values are computed remotely and transmitted to the vehicles over noisy wireless networks. The interference and fading effects in the noisy network often lead to data losses \cite{Mishra2018}. This {aspect} further leads to the requirement of designing scheduling {sequences} that preserve good qualitative properties of the plants {in the presence of data losses}. Our objective is to address this design challenge. {Typically, the existence of a favourable scheduling sequence depends not only on the parameters of the shared network but also on the plant dynamics. This feature motivates the problem of designing controllers for the plants such that the plants and the communication network together admit a desired scheduling sequence. We also address the design of static state-feedback controllers such that certain good qualitative properties of all plants in an NCS are preserved under scheduling.}
\subsection{Prior works}
\label{ss:lit_survey}
    The existing classes of scheduling sequences can be classified broadly into two categories: \emph{static (also called {periodic}, {fixed}, or {open-loop})} and \emph{dynamic (also called {non-periodic}, or {closed-loop})}. In case of the former, a finite length allocation scheme of the network is determined offline and is applied eternally in a periodic manner, while in case of the latter, the allocation of the shared network is determined based on some information about the plant (e.g., states, outputs, access status of sensors and actuator, etc.). For NCSs with continuous-time linear plants, static scheduling sequences that preserve stability of all plants under ideal communication, are characterized using common Lyapunov functions in \cite{Hristu2001} and piecewise Lyapunov-like functions with average dwell time switching in \cite{Lin2005}. A more general case of co-designing a static scheduling sequence and control action is addressed for ideal communication using combinatorial optimization with periodic control theory in \cite{Rehbinder2004} and for delayed communication using Linear Matrix Inequalities (LMIs) optimization with average dwell time technique in \cite{Dai2010}. The authors of \cite{Zhang2006} characterize static scheduling sequences that ensure reachability and observability of the plants under limited but ideal communication, and design an observer-based feedback controller for these sequences. The corresponding techniques were later extended to the case of constant transmission delays \cite{Hristu2008} and Linear Quadratic Gaussian control \cite{Hristu_Zhang2008}. Event-triggered dynamic scheduling sequences that preserve stability of all plants under communication delays are proposed in \cite{Al-Areqi'15}. In \cite{Quevedo2014} the authors propose a mechanism to allocate network resources by finding optimal node that minimizes a certain cost function in every network time instant. The design of dynamic scheduling sequences for stability of each plant under both communication uncertainties and computational limitations is studied in \cite{Saha2015}. In \cite{Gatsis2016} a class of distributed control-aware random network access sequences for sensors such that all control loops are stabilizable, is presented. {A reinforcement learning based sensor scheduling sequence for Cyber-Physical Systems is presented in \cite{Leong2020}}. A dynamic scheduling sequence based on predictions of both control performance and channel quality at run-time, is proposed in \cite{Ma2019}. {In \cite{Liu2019} the authors present a model predictive control scheme for scheduling and control co-design of NCSs.}

    {Periodic scheduling sequences are easier to implement, often near optimal, and guarantee activation of each sensor and actuator, see \cite{Peters'16,Longo, Hristu'05} for detailed discussions. They are preferred for safety-critical control systems \cite[\S2.5.1]{Longo}. It is also observed in \cite{Orihuela'14, Peters'16} that periodic phenomenon appears in non-periodic schedules. Recently in \cite{quevedo2020} we employed a blend of multiple Lyapunov-like functions and graph theory to design stability preserving periodic scheduling sequences under ideal communication. In this paper we study periodic scheduling sequences under communication uncertainties.}
\subsection{Our contributions}
\label{ss:contri}
    We consider an NCS consisting of multiple discrete-time linear plants whose feedback loops are closed through a shared communication network, {a pictorial representation is given in Figure \ref{fig:ncs}}. We assume that the plants are unstable in open-loop and stable {when controlled} in closed-loop. Due to a limited communication capacity of the network, only a {subset of the} plants can exchange information with their controllers at any instant of time. Consequently, the remaining plants operate in open-loop {potentially leading to instability}. In addition, the communication network is prone to data losses. In particular, at any time instant, the control input is lost in a channel with a known probability. If the control input is lost in a channel at a time instant, then the plant accessing that channel at that instant also operates in open-loop.
    \begin{figure}[htbp]
    	\begin{center}
	\scalebox{0.6}{
	\begin{tikzpicture}[every path/.style={>=latex},base node/.style={draw,rectangle, scale = 1.4}]
	\node[base node] (a) at (-2,5) {Controller 1};
	\node[base node] (b) at (3.5,4) {Plant 1};
	\node[base node] (c) at (-2,2) {Controller 2};
	\node[base node] (d) at (3.5,1) {Plant 2};
	\node[base node] (e) at (-2,-2) {Controller N};
	\node[base node] (f) at (3.5,-3) {Plant N};	
	
	\draw (-4.5 ,5) edge (a);
	\draw (-4.5,5) edge (-4.5,4);
	\draw[->] (-4.5,4) -- (0,4);
	\draw (a) edge (0.4,5);
	\draw[->] (b) -- (5.5,4);
	\draw (5.5,4) edge (5.5,5);
	\draw[->] (1,4) -- (b);
	\draw[-.] (1,4) -- (0.4,4.5);
	\draw (5.5,5) edge (1.1,5);
	\draw[-.] (1.1,5) -- (0.6,5.5);

	\draw (-4.5 ,2) edge (c);
	\draw (-4.5,2) edge (-4.5,1);
	\draw[->] (-4.5,1) -- (0,1);
	\draw (c) edge (0.4,2);
	\draw[->] (d) -- (5.5,1);
	\draw (5.5,1) edge (5.5,2);
	\draw[->] (1,1) -- (d);
	\draw[-.] (1,1) -- (0.4,1.5);
	\draw (5.5,2) edge (1.1,2);
	\draw[-.] (1.1,2) -- (0.6,2.5);

	\draw (-4.5 ,-2) edge (e);
	\draw (-4.5,-2) edge (-4.5,-3);
	\draw[->] (-4.5,-3) -- (0,-3);
	\draw (e) edge (0.4,-2);
	\draw[->] (f) -- (5.5,-3);
	\draw (5.5,-3) edge (5.5,-2);
	\draw[->] (1,-3) -- (f);
	\draw[-.] (1,-3) -- (0.4,-2.5);
	\draw (5.5,-2) edge (1.1,-2);
	\draw[-.] (1.1,-2) -- (0.6,-1.5);

	\draw[dashed] (-0.4,-4) -- (-0.4,6);
	\draw[dashed] (2,-4) -- (2,6);
	\draw[dashed] (-0.4,-4) -- (2,-4);
	\draw[dashed] (-0.4,6) -- (2,6);

	\node (g) at (0.75,-4.5) {Communication network};
	\node (h) at (3.5,-0.5) {\(\vdots\)};

	\end{tikzpicture}
	}
	\caption{Block diagram of NCS}\label{fig:ncs}
	\end{center}
    \end{figure}
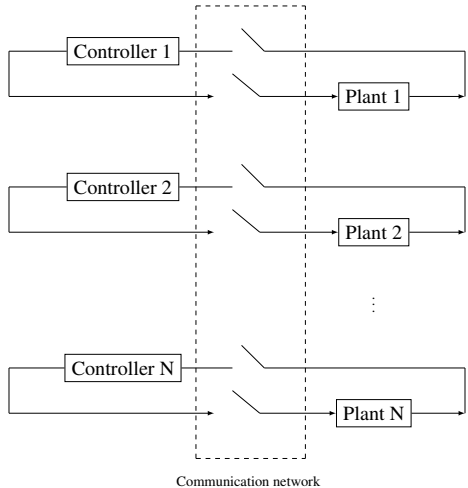

    {Our first objective is to design periodic scheduling sequences that
preserve exponential second moment stability (ESMS) of each plant in the NCS.} We model the individual plants of an NCS
as switched systems, whose subsystems are the open-loop (unstable mode) and closed-loop (stable mode) operations of the
plants and the switching signals are time-inhomogeneous Markov chains. The open-loop operation of a plant occurs at the
time instants when it does not have access to the shared network, or when it has access to the shared network but the control
input is lost in transit. The closed-loop operation occurs at the time instants when the plant has access to the shared network
and the control input is received. Clearly, a switching signal is governed by both the scheduling sequence and the data loss
signal. We design scheduling sequences in a manner such that each switching signal is stabilizing under all data loss signals. Towards this end, we employ the following steps:
	{\begin{itemize}[label = \(\circ\), leftmargin = *]
        \item First, we apply off-the-shelf stability conditions for switched linear systems whose switching signals are time-inhomogeneous Markov chains with periodic transition probability matrices from \cite{Fang2002} to arrive at necessary and sufficient conditions involving the open-loop and closed-loop dynamics of the plants, the capacity of the network and the probability of data loss, under which a periodic scheduling sequence is stabilizing. {The proposed conditions rely on the existence of sets of symmetric and positive definite matrices that satisfy certain equalities.}
        \item Second, given the open-loop and closed-loop dynamics of the plants, the capacity of the shared network, the probability of data loss, a period for the scheduling sequence, we present an algorithm that verifies our stability conditions by solving a set of feasibility problems that involves LMIs. If these feasibility problems admit solutions, then the algorithm designs stabilizing periodic scheduling sequences. Otherwise, the algorithm reports non-existence of a stabilizing periodic scheduling sequence with the period under consideration.
    \end{itemize}}
    {Our second objective is to design static state-feedback controllers for the individual plants in an NCS such that the plants and the shared communication network together satisfy our stability conditions. Towards this end, we employ the following step:
    \begin{itemize}[label = \(\circ\), leftmargin = *]
    	  {\item Given the open-loop dynamics of the plants, the parameters of the network, a period for the scheduling sequence, we present an algorithm that designs state-feedback controllers such that the set of feasibility problems, mentioned above, admits solutions. The algorithm involves solving certain sets of matrix inequalities to determine sets of matrices with some symmetric and positive definite elements. The state-feedback controllers are designed as functions of these matrices.}
   \end{itemize}
    {Our design of controllers is with respect to a scheduling sequence design mechanism fixed a priori and hence, unlike \cite{Rehbinder2004,Dai2010}, {we} do not address ``co''-design of scheduling sequences and feedback controllers.}
    }

    {Our results are demonstrated on a set of numerical experiments.}
\subsection{Paper organization}
\label{ss:paper_org}
    The remainder of this paper is organized as follows: in \S\ref{s:prob_stat} we formulate the problem under consideration. The apparatus for our design of scheduling sequences and analysis of stability are described in \S\ref{s:prelims}. Our {main technical} results appear in \S\ref{s:mainres}. We also describe various features of our results in this section. Numerical experiments are presented in \S\ref{s:num_ex}. We conclude in \S\ref{s:concln} with a brief discussion of future research directions. {Proofs of our results are presented in a consolidated manner in \S\ref{s:proofs}.}
\subsection{Notation}
\label{ss:notation}
    Standard notations and terminologies are employed throughout the paper. \(\N\) is the set of natural numbers, \(\N_{0} = \N\cup\{0\}\), and \(\R\) is the set of real numbers. For a finite set \(A\), we employ \(\abs{A}\) to denote its cardinality, i.e., the number of elements in \(A\). Two finite sets \(A\) and \(B\) are distinct, if there exists at least one element \(i\) such that \(i\in A\) but \(i\notin B\), or vice-versa. For \(a\in\R\), \(\lfloor a \rfloor\) denotes the largest integer less than or equal to \(a\). For symmetric block matrices, \(\star\) acts as ellipsis for the terms that are introduced by symmetry. \(\diag\{M_1,M_2,\ldots,M_n\}\) denotes a block-diagonal matrix with diagonal elements \(M_1\), \(M_2,\ldots\), \(M_n\). We will operate in a probability space \((\Omega,\mathcal{F},\PP)\), where \(\Omega\) is the sample space, \(\mathcal{F}\) is the \(\sigma\)-algebra of events and \(\PP\) is the probability measure.
\section{Problem statement}
\label{s:prob_stat}
    We consider an NCS with \(N\) plants whose dynamics are given by
    \begin{align}
    \label{e:plant_i}
        x_{i}(t+1) = A_{i}x_{i}(t) + B_{i}u_{i}(t),\:\:x_{i}(0) = x_{i}^{0},\:\:t\in\N_{0},
    \end{align}
    where \(x_{i}(t)\in\R^{d}\) and \(u_{i}(t)\in\R^{m}\) are the vectors of states and inputs of the \(i\)-th plant at time \(t\), respectively, \(i=1,2,\ldots,N\). Each plant \(i\) employs a state-feedback controller, \(u_{i}(t) = K_{i}x_{i}(t)\), \(t\in\N_{0}\). The matrices, \(A_{i}\in\R^{d\times d}\), \(B_{i}\in\R^{d\times m}\) and \(K_{i}\in\R^{m\times d}\), \(i=1,2,\ldots,N\), are constant.

    \begin{assump}
    \label{a:stability}
    \rm{
        The open-loop dynamics of each plant is unstable and each controller is stabilizing. More specifically, the matrices, \(A_{i}+B_{i}K_{i}\), \(i=1,2,\ldots,N\), are Schur stable and the matrices, \(A_{i}\), \(i=1,2,\ldots,N\), are unstable.\footnote{A matrix \(A\in\R^{d\times d}\) is Schur stable if all its eigenvalues are inside the open unit disk. We call \(A\) unstable if it is not Schur stable.}
    }
    \end{assump}

    The controllers are remotely located and each plant communicates with its controller through a shared communication network. We consider that the shared network has the following properties:
    \begin{itemize}[label = \(\circ\), leftmargin = *]
        \item It has a limited communication capacity in the sense that at any time instant, only \(M\) plants \((0<M<N)\) can access the network. Consequently, the remaining \(N-M\) plants operate in open-loop.
        \item The communication channels from the controllers to the plants are prone to data losses. In particular, at any time instant, the control input is lost in a channel with probability \(p\). The plant accessing this channel at that instant also operates in open-loop.
    \end{itemize}

    In view of Assumption \ref{a:stability} and the properties of the shared communication network, each plant in \eqref{e:plant_i} operates in two modes:
    \begin{enumerate}[label = (\alph*), leftmargin = *]
        \item stable mode (closed-loop operation) when the plant has access to the shared communication network and its control input is received, and
        \item unstable mode (open-loop operation) when the plant does not have access to the shared communication network, or it has access to the shared network but its control input is lost in the channel.
    \end{enumerate}

    We let \(\is\) and \(\iu\) denote the mode of operation of the \(i\)-th plant when it has access to the shared network and its control input is received (stable) and when it has access to the shared network but its control input is lost or when it does not have access to the shared network (unstable), respectively, \(A_{\is} = A_i+B_iK_i\) and \(A_{\iu} = A_i\), \(i=1,2,\ldots,N\). Let \(\kappa_{j}:\N_{0}\to\{0,1\}\) denote the data loss signal at the \(j\)-th channel of the communication network, \(j=1,2,\ldots,M\). If \(\kappa_{j}(t) = 0\), then the plant accessing the \(j\)-th channel at time \(t\) receives its control input, and if \(\kappa_{j}(t) = 1\), then the plant accessing the \(j\)-th channel at time \(t\) operates in open-loop. We have
    \begin{align}
    \label{e:data_loss}
        \kappa_{j}(t) &=
        \begin{cases}
            1,\:\:\text{with probability}\:p,\\
            0,\:\:\text{with probability}\:1-p,
        \end{cases}
        \:\:t\in\N_{0}.
    \end{align}
    Let
    \[
        \Svec = \{s\in\{1,2,\ldots,N\}^{M}\:|\:\text{the elements of \(s\) are distinct}\}
    \]
    be the set of vectors that consist of \(M\) distinct elements from the set \(\{1,2,\ldots,N\}\). We call a function \(\gamma:\N_{0}\to\Svec\) that specifies, at every time \(t\), \(M\) plants of the NCS which access the shared network at that time, as a \emph{scheduling sequence}. {We shall restrict our attention to scheduling sequences that are purely time-dependent in the sense that the choice of \(\gamma(t) = s\in\Svec\) at any instant of time \(t\) depends solely on \(t\) and is not governed by any information about the plants (e.g., states, outputs, access status of sensors and actuators, etc.) or the shared network (e.g., history of data losses, etc.). In addition, we shall focus on the class of \(\gamma\) that admits a periodic structure, the mathematical definition of which will follow momentarily.}

    Let \(R_i = \{\is,\iu\}\), \(i=1,2,\ldots,N\). The dynamics of each plant \(i\in\{1,2,\ldots,N\}\) under scheduling and data loss can be modelled as follows:
    \begin{align}
    \label{e:plant_model}
        x_{i}(t+1) = A_{\sigma_{i}(t)}x(t),\:\:x_i(0) = x_{i}^{0},\:\:\sigma_i(t)\in R_i,\:\:t\in\N_{0}.
    \end{align}
    Notice that \eqref{e:plant_model} is a switched linear system whose set of subsystems is \(R_i\) and a switching signal \(\sigma_i:\N_0\to R_i\) satisfies
    \begin{align}
    \label{e:sw_sig}
        \sigma_i(t) =
        \begin{cases}
            \is,&\:\:\text{if \(i\) is an element of \(\gamma(t)\) and the channel that \(i\) is accessing does not suffer}\\&\:\:\text{from data loss at time \(t\)},\\
            \iu,&\:\:\text{if \(i\) is not an element of \(\gamma(t)\), or \(i\) is an element of \(\gamma(t)\) and the channel that}\\&\:\:\text{\(i\) is accessing suffers from data loss at time \(t\).}
        \end{cases}
    \end{align}
    We observe that  \(\sigma_i\) is a Markov chain, defined on \((\Omega,\mathcal{F},\PP)\), taking values in \(R_i\) with transition probability matrix, \(\Pi_i(t) = (p_{k\ell})_{k,\ell\in R_i}\), \(t\in\N_0\) and initial probability distribution, \(\Phi_{i0} = \pmat{\phi_{i0}(\is) & \phi_{i0}(\iu)}\). Here,
    \begin{align*}
        p_{\is\is}(t) &= \PP(\sigma_i(t) = \is\:|\:\sigma_i(t-1) = \is) =
        \begin{cases}
            1-p,&\:\:\text{if \(i\) is an element of \(\gamma(t)\)},\\
            0,&\:\:\text{if \(i\) is not an element of \(\gamma(t)\)},\\
        \end{cases}
    \end{align*}
    \begin{align*}
        p_{\is\iu}(t) &= \PP(\sigma_i(t) = \iu\:|\:\sigma_i(t-1) = \is) =
        \begin{cases}
            p,&\:\:\text{if \(i\) is an element of \(\gamma(t)\)},\\
            1,&\:\:\text{if \(i\) is not an element of \(\gamma(t)\)},\\
        \end{cases}
    \end{align*}
    \begin{align*}
        p_{\iu\is}(t) &= \PP(\sigma_i(t) = \is\:|\:\sigma_i(t-1) = \iu) =
        \begin{cases}
            1-p,&\:\:\text{if \(i\) is an element of \(\gamma(t)\)},\\
            0,&\:\:\text{if \(i\) is not an element of \(\gamma(t)\)},\\
        \end{cases}
    \end{align*}
    \begin{align*}
        p_{\iu\iu}(t) &= \PP(\sigma_i(t) = \iu\:|\:\sigma_i(t-1) = \iu) =
        \begin{cases}
            p,&\:\:\text{if \(i\) is an element of \(\gamma(t)\)},\\
            1,&\:\:\text{if \(i\) is not an element of \(\gamma(t)\)},\\
        \end{cases}
    \end{align*}
    and
    \begin{align*}
        \phi_{i0}(\is) &= \PP(\sigma_i(0) = \is) =
        \begin{cases}
            1-p,&\:\:\text{if \(i\) is an element of \(\gamma(0)\)},\\
            0,&\:\:\text{if \(i\) is not an element of \(\gamma(0)\)},
        \end{cases}
    \end{align*}
    \begin{align*}
        \phi_{i0}(\iu) &= \PP(\sigma_i(0) = \iu) =
        \begin{cases}
            p,&\:\:\text{if \(i\) is an element of \(\gamma(0)\)},\\
            1,&\:\:\text{if \(i\) is not an element of \(\gamma(0)\)}.
        \end{cases}
    \end{align*}

    \begin{lemma}
    \label{lem:auxres1}
        Fix \(i\in\{1,2,\ldots,N\}\). Then \(\Pi_i(t)\) and \(\Phi_{i0}\) are well-defined.
    \end{lemma}

    {At this point, it is worth noting that the model for \(\sigma_i\) described above is, in principle, a controlled Markov chain, which under a purely time-dependent scheduling sequence $\gamma$, admits the properties of a time-inhomogeneous Markov chain. In particular, for a periodic scheduling sequence $\gamma$, we obtain a time-inhomogeneous Markov chain with periodic transition probability matrix \cite[Chapter 1]{Suhov}, see also Lemma \ref{lem:auxres2}. }

    \begin{example}
    \label{ex:model}
    \rm{
        Let \(N = 2\) and \(M = 1\). Suppose that
        \(
            \gamma = \pmat{1},\pmat{2},\pmat{1},\pmat{2},\pmat{1},\ldots.
        \)
        Then the dynamics of plant \(1\) can be modelled as
        \(
            x_1(t+1) = A_{\sigma_1(t)}x_1(t),
        \)
        where \(\sigma_1\) is a Markov chain with
        \begin{align*}
            \Pi_1(t) &=
            \begin{cases}
                \pmat{1-p & p\\1-p & p},&\:\:t=2,4,6,\ldots,\\
                \pmat{0 & 1\\0 & 1},&\:\:t=1,3,5,\ldots,
            \end{cases}
        \end{align*}
        and initial distribution \(\Phi_{10} = \pmat{1-p & p}\). The dynamics of plant \(2\) can be modelled as
        \(
            x_2(t+1) = A_{\sigma_2(t)}x_2(t),
        \)
        where \(\sigma_2\) is a Markov chain with
        \begin{align*}
            \Pi_2(t) &=
            \begin{cases}
                \pmat{1-p & p\\1-p & p},&\:\:t=1,3,5,\ldots,\\
                \pmat{0 & 1\\0 & 1},&\:\:t=2,4,6,\ldots,
            \end{cases}
        \end{align*}
        and initial distribution \(\Phi_{20} = \pmat{0 & 1}\).
    }
    \end{example}

    \begin{defn}
    \label{d:stability}
    \rm{
        The \(i\)-th plant in the NCS is \emph{exponentially second moment stable (ESMS)} under a scheduling sequence \(\gamma\), if for any initial state \(x_i(0)\in\R^{d}\) and any admissible initial distribution \(\Phi_{i0}\), there exist constants {\(\alpha_i,\beta_i>0\)}, independent of \(x_i(0)\) and \(\Phi_{i0}\), such that
        \begin{align}
        \label{e:stability}
            {\EE\{\norm{x_i(t)}^{2}\}\leq\alpha_i\norm{x_i(0)}^{2}e^{-\beta_i t}\:\:\text{for all}\:t\in\N_0}.
        \end{align}
        }
    \end{defn}

    We will first solve the following problem:
    \begin{prob}
    \label{prob:main}
    \rm{
        Given the matrices, \(A_{i}\), \(B_{i}\), \(K_{i}\), \(i=1,2,\ldots,N\), the capacity of the network, \(M\),  and the probability of data loss, \(p\), design a {purely time-dependent periodic} scheduling sequence, \(\gamma\), that ensures stability of each plant \(i\) in \eqref{e:plant_i} in the sense of Definition \ref{d:stability}.
        }
    \end{prob}

    In the sequel we will call a \(\gamma\) that is a solution to Problem \ref{prob:main} as a stabilizing scheduling sequence. {Purely time-dependent periodic scheduling sequences are easier to implement compared to scheduling sequences that rely on information about the plants (e.g., states, outputs, access status of sensors and actuators, etc.) and/or the shared network (e.g., history of data losses, etc.). We will demonstrate stabilizing properties of such sequences. In particular,} towards solving Problem \ref{prob:main}, we will employ the following steps: First, we identify necessary and sufficient conditions involving the matrices, \(A_{\is}\), \(A_{\iu}\), \(i=1,2,\ldots,N\), the capacity of the network, \(M\), and the data loss probability, \(p\), under which there exists a stabilizing periodic scheduling sequence. Second, we present an algorithm that verifies our stability conditions for pre-specified period and if satisfied, designs such a sequence. Otherwise, the algorithm reports non-existence of a stabilizing periodic scheduling sequence with the given period. {We will then solve the following problem:
     \begin{prob}
    \label{prob:next}
    \rm{
        Given the matrices, \(A_{i}\), \(B_{i}\), \(i=1,2,\ldots,N\), the capacity of the network, \(M\), the probability of data loss, \(p\), and a period, \(\ell\) for the scheduling sequence, design state-feedback controller matrices \(K_i\), \(i=1,2,\ldots,N\) such that there exists a {purely time-dependent periodic} scheduling sequence, \(\gamma\), that ensures stability of each plant \(i\) in \eqref{e:plant_i} in the sense of Definition \ref{d:stability}.
        }
    \end{prob}
   Towards solving Problem \ref{prob:next}, we present an algorithm that designs \(K_i\), \(i=1,2,\ldots,N\) based on the given set of information. Our algorithm relies on sufficient conditions on the matrices \(A_i,B_i\), the capacity of the network, \(M\), the probability of data loss, \(p\), and the period, \(\ell\) of a scheduling sequence for the existence of \(K_i\), and no solution to this algorithm does not imply the non-existence of state-feedback controllers such that the plants and the shared network together admit a stabilizing periodic scheduling sequence.}

   Prior to presenting our results, we catalog a set of preliminaries that our solutions to Problems \ref{prob:main} and \ref{prob:next} will rely on.
\section{Preliminaries}
\label{s:prelims}
    \begin{defn}
    \label{d:periodic-sequence}
    \rm{
        A scheduling sequence, \(\gamma\), is \emph{periodic} if there exists \(\ell\in\N\) such that \(\gamma(t) = \gamma(t+\ell)\) for all \(t\in\N_0\). We call \(\ell\) to be the \emph{period} of \(\gamma\).
    }
    \end{defn}
    \begin{defn}
    \label{d:periodic-mat}
    \rm{
        Consider \(i\in\{1,2,\ldots,N\}\). The transition probability matrix, \(\Pi_i\), is \emph{periodic} if there exists \(m\in\N\) such that \(\Pi_i(t) = \Pi_i(t+m)\) for all \(t\in\N\). We call \(m\) to be the period of \(\Pi_i\).
        }
    \end{defn}
    \begin{lemma}
    \label{lem:auxres2}
         The following are equivalent:
         \begin{enumerate}[label = (\roman*), leftmargin = *]
            \item A scheduling sequence, \(\gamma\), is periodic with period \(\ell\).
            \item For each plant \(i\in\{1,2,\ldots,N\}\), the transition probability matrix, \(\Pi_i\), is periodic with period \(\ell\).
         \end{enumerate}
    \end{lemma}

    We recall
    \begin{theorem}{\cite[Theorem 2.3]{Fang2002}}
    \label{t:recall_res}
    \rm{
        Consider a plant \(i\). Suppose that the transition probability matrix, \(\Pi_i(t)\), is periodic with period \(\ell\). Then the following are equivalent:
        \begin{enumerate}[label = \roman*), leftmargin = *]
            \item For some symmetric and positive definite matrices, \(Q_{i,1}(j)\), \(Q_{i,2}(j),\ldots\), \(Q_{i,\ell}(j)\), \(j\in R_i\), there exist symmetric and positive definite matrices, \(P_{i,1}(j)\), \(P_{i,2}(j),\ldots\), \(P_{i,\ell}(j)\), \(j\in R_i\), such that
                \begin{align}
                    \label{e:maincondn1}\sum_{j\in R_i}p_{kj}(\tau)A_k^\top P_{i,\tau+1}(j)A_k - P_{i,\tau}(k) &= -Q_{i,\tau}(k),\:\:\tau = 1,2,\ldots,\ell-1,\:\:k\in R_i,\\
                    \intertext{and}
                    \label{e:maincondn2}\sum_{j\in R_i}p_{kj}(\ell)A_k^\top P_{i,1}(j)A_k-P_{i,\ell}(k) &= -Q_{i,\ell}(k),\:\:k\in R_i.
                \end{align}
            \item The plant \(i\) is ESMS.
        \end{enumerate}
    }
    \end{theorem}

    Lemma \ref{lem:auxres2} ensures the equivalence between periodicity of a scheduling sequence, \(\gamma\), and periodicity of transition probability matrix, \(\Pi_i\), for each plant \(i\), while Theorem \ref{t:recall_res} provides necessary and sufficient conditions for ESMS of plant \(i\) with a periodic transition probability matrix. We will employ Lemma \ref{lem:auxres2} and Theorem \ref{t:recall_res} in our design of stabilizing periodic scheduling sequences.

    We are now in a position to present our solutions to Problems \ref{prob:main} and \ref{prob:next}.
\section{Results}
\label{s:mainres}
\subsection{Design of stabilizing periodic scheduling sequences}
\label{ss:results_set1}
	{We first identify necessary and sufficient conditions on the plant dynamics that ensure existence of periodic scheduling sequences under which stability of each plant is preserved for all admissible data loss signals.}
    \begin{theorem}
    \label{t:mainres}
        Consider an NCS described in \S\ref{s:prob_stat}. There exists a periodic scheduling sequence, \(\gamma\), with period \(\ell\in\N\), that ensures ESMS of each plant \(i\in\{1,2,\ldots,N\}\), if and only if there exist distinct sets \(\D_q\subset\{1,2,\ldots,N\}\), \(q=1,2,\ldots,\ell\), that satisfy the following conditions:
        \begin{enumerate}[label = \((C\arabic*)\), leftmargin = *]
            \item\label{mainprop1} \(\abs{\D_q} = M\), \(q=1,2,\ldots,\ell\),
            \item\label{mainprop2} \(\displaystyle{\bigcup_{q=1}^{\ell}\D_q} = \{1,2,\ldots,N\}\),
            \item\label{mainprop3} each element \(i\in\D_q\), \(q=1,2,\ldots,\ell\), satisfies that for some symmetric and positive definite matrices \(Q_{i,1}(j)\), \(Q_{i,2}(j),\ldots\), \(Q_{i,\ell}(j)\), \(j\in R_i\), there exist symmetric and positive definite matrices \(P_{i,1}(j)\), \(P_{i,2}(j),\ldots\), \(P_{i,\ell}(j)\), \(j\in R_i\), such that
                \begin{align}
                    \label{e:condn1}(1-p)A_{k}^\top P_{i,\tau+1}(\is)A_k + pA_k^\top P_{i,\tau+1}(\iu)A_k - P_{i,\tau}(k) &= -Q_{i,\tau}(k),\:\:\text{if}\:i\in\D_{\tau+1},\:\tau=1,2,\ldots,\ell-1,\:k\in R_i,\\
                    \label{e:condn2}A_k^\top P_{i,\tau+1}(\iu)A_k - P_{i,\tau}(k) &= -Q_{i,\tau}(k),\:\:\text{if}\:i\notin\D_{\tau+1},\:\tau = 1,2,\ldots,\ell-1,\:k\in R_i,
                \end{align}
                and
                \begin{align}
                    \label{e:condn3}(1-p)A_k^\top P_{i,1}(\is)A_k + pA_k^\top P_{i,1}(\iu)A_k - P_{i,\ell}(k) &= -Q_{i,\ell}(k),\:\:\text{if}\:i\in\D_1,\:k\in R_i,\\
                    \label{e:condn4}A_k^\top P_{i,1}(\iu)A_k - P_{i,\ell}(k) &= -Q_{i,\ell}(k),\:\:\text{if}\:i\notin\D_1,\:k\in R_i.
                \end{align}
        \end{enumerate}
    \end{theorem}

    {Theorem \ref{t:mainres} provides necessary and sufficient conditions for the existence of stabilizing scheduling sequences that are periodic with period \(\ell\). It relies on the existence of \(\ell\) distinct subsets, \(\D_q\), of \(\{1,2,\ldots,N\}\), that satisfy conditions \ref{mainprop1}-\ref{mainprop3}. Condition \ref{mainprop1} ensures that each set \(\D_q\), \(q=1,2,\ldots,\ell\), has \(M\)-many elements, condition \ref{mainprop2} ensures that each element of the set \(\{1,2,\ldots,N\}\) appears in at least one \(\D_q\), \(q\in\{1,2,\ldots,\ell\}\), and condition \ref{mainprop3} ensures that for each \(i\in\{1,2,\ldots,N\}\), the matrices \(A_{\is}\), \(A_{\iu}\), \(i=1,2,\ldots,N\) and the probability of data loss, \(p\), together satisfy certain matrix equalities based on which \(\D_q\) the plant \(i\) appears in, \(q\in\{1,2,\ldots,\ell\}\).}

     {Given the matrices \(A_i\), \(B_i\), \(K_i\), \(i=1,2,\ldots,N\), the capacity of the network, \(M\), the probability of data loss, \(p\), and a number \(\ell\in\N\), we next present an algorithm that checks the conditions of Theorem \ref{t:mainres} and if satisfied, designs a purely time-dependent periodic scheduling sequence, \(\gamma\), whose period is \(\ell\).} Such a scheduling sequence allows the elements of \(\D_q\), \(q=1,2,\ldots,\ell\), to access the shared communication network (in order) and repeats this process eternally.
    \begin{algorithm}[htbp]
			\caption{Design of stabilizing periodic scheduling sequences} \label{algo:sched-policy_design}
		\begin{algorithmic}[1]
			\renewcommand{\algorithmicrequire}{\textbf{Input:}}
			\renewcommand{\algorithmicensure}{\textbf{Output:}}
			
			\REQUIRE The matrices \(A_i\), \(B_i\), \(K_i\), \(i=1,2,\ldots,N\), the capacity of the network, \(M\), the probability of data loss, \(p\) and a number \(\ell\in\N\)
			\ENSURE A stabilizing periodic scheduling sequence, \(\gamma\), with period \(\ell\) or a failure message

            \STATE Compute
            \begin{align*}
                \ell_{\min} :=
                \begin{cases}
                    \lfloor{N/M}\rfloor,\:\:&\:\text{if}\:N\%M = 0,\\
                    \lfloor{N/M}\rfloor + 1,\:\:&\:\text{if}\:N\%M \neq 0.
                \end{cases}
            \end{align*}
            \IF {\(\ell<\ell_{\min}\),}
                \STATE Output ``Error''.
            \ELSE
			
			 \STATE Construct \(\Gamma = \biggl\{(\D_q)_{q=1}^{\ell}\:|\:\D_q,\:q=1,2,\ldots,\ell\:\text{are distinct and satisfy \ref{mainprop1}-\ref{mainprop2}}\biggr\}\).

             \FOR {each \((\D_q)_{q=1}^{\ell}\in\Gamma\)}
                \FOR {each \(i\in\D_q\), \(q=1,2,\ldots,\ell\)}
                    \STATE Solve the following feasibility problem for \(P_{i,1}(j)\), \(P_{i,2}(j),\ldots\), \(P_{i,\ell}(j)\in\R^{d\times d}\), \(j\in R_i\):
                    \begin{align}
                    \label{e:feasprob}
                        \minimize \:\:&\:\: 1\nonumber\\
                        \sbjto \:\:&\:\:
                        \begin{cases}
                            &(1-p)A_{k}^\top P_{i,\tau+1}(\is)A_k + pA_k^\top P_{i,\tau+1}(\iu)A_k - P_{i,\tau}(k) \prec 0,\\&\quad\quad\:\text{if}\:i\in\D_{\tau+1},\:\tau=1,2,\ldots,\ell-1,\:k\in R_i,\\
                            &A_k^\top P_{i,\tau+1}(\iu)A_k - P_{i,\tau}(k) \prec 0,\\&\quad\quad\:\text{if}\:i\notin\D_{\tau+1},\:\tau = 1,2,\ldots,\ell-1,\:k\in R_i,\\
                            &(1-p)A_k^\top P_{i,1}(\is)A_k + pA_k^\top P_{i,1}(\iu)A_k - P_{i,\ell}(k) \prec 0,\\&\quad\quad\:\:\text{if}\:i\in\D_1,\:k\in R_i,\\
                            &A_k^\top P_{i,1}(\iu)A_k - P_{i,\ell}(k) \prec 0,\\&\quad\quad\:\:\text{if}\:i\notin\D_1,\:k\in R_i,\\
                            &P_{i,1}(j) = P_{i,1}^\top(j), P_{i,2}(j)=P_{i,2}^\top(j),\ldots, P_{i,\ell}(j) = P_{i,\ell}^\top(j),\:\:j\in R_i,\\
                            &P_{i,1}(j), P_{i,2}(j),\ldots, P_{i,\ell}(j)\succ 0,\:\:j\in R_i.
                        \end{cases}
                    \end{align}
                \ENDFOR
                \IF {\eqref{e:feasprob} admits a solution for each \(i\in\D_q\), \(q=1,2,\ldots,\ell\),}
                    \STATE Go to \ref{step:design}.
                \ELSE
                    \STATE Output ``No stabilizing periodic scheduling sequence of period \(\ell\)''.
                \ENDIF
             \ENDFOR

             \STATE\label{step:design} Set \(v_q\in\Svec\) to be the vector containing the elements in the set \(\D_q\), \(q=1,2,\ldots,\ell\).
             \STATE Initialize \(\gamma(0) = v_1\), \(\gamma(1) = v_2,\ldots\), \(\gamma(\ell-1) = v_\ell\).
             \STATE Set \(\gamma(t+\ell) = \gamma(t)\) for all \(t\in\N_0\).
        \ENDIF
		\end{algorithmic}
	\end{algorithm}

    Recall that the plants under consideration are open-loop unstable. Consequently, any stabilizing scheduling sequence must allow all plants \(i\in\{1,2,\ldots,N\}\) to access the shared communication network. Algorithm \ref{algo:sched-policy_design} first computes \(\ell_{\min}\) --- the minimum number of distinct sets required to facilitate the above. If the desired period of \(\gamma\) is less than \(\ell_{\min}\), then the algorithm outputs an error message. Otherwise, for every possible choices of \((\D_q)_{q=1}^{\ell}\) that satisfy \ref{mainprop1}-\ref{mainprop2}, Algorithm \ref{algo:sched-policy_design} solves the feasibility problem \eqref{e:feasprob} to determine, if exist, the matrices \(P_{i,1}(j)\), \(P_{i,2}(j),\ldots\), \(P_{i,\ell}(j)\), \(j\in R_i\), \(i=1,2,\ldots,N\), that satisfy conditions \eqref{e:condn1}-\eqref{e:condn4}. The feasibility problem \eqref{e:feasprob} can be solved by using standard LMI solver toolboxes. If there exists \((\D_q)_{q=1}^{\ell}\) such that a solution to \eqref{e:feasprob} is obtained for each \(i\in\D_q\), \(q=1,2,\ldots,\ell\), then Algorithm \ref{algo:sched-policy_design} designs a scheduling sequence, \(\gamma\), with period \(\ell\). If for all \((\D_q)_{q=1}^{\ell}\), there exists \(i\in\D_q\), \(q\in\{1,2,\ldots,\ell\}\), such that \eqref{e:feasprob} does not admit a solution, then Algorithm \ref{algo:sched-policy_design} reports non-existence of a stabilizing periodic scheduling sequence of period \(\ell\). {Given the matrices \(A_i\), \(B_i\), \(K_i\), \(i=1,2,\ldots,N\), the probability of data loss, \(p\) and a period, \(\ell\) of a scheduling sequence, the feasibility problem \eqref{e:feasprob} can be solved by using standard LMI solver toolboxes. We will employ the LMI solver toolbox in MATLAB for our numerical examples presented in \S\ref{s:num_ex}. The underlying algorithm of this toolbox is Gahinet and Nemirovski's projective method \cite{Gahinet1997}, which has a polynomial time complexity.}
     \begin{example}
    \label{ex:algorithm}
    \rm{
        Let \(N = 10\), \(M = 4\) and \(\ell = 3\). Clearly, \(\ell = \ell_{\min}\). Suppose that \(\D_1 = \{1,3,4,7\}\), \(\D_2 = \{2,4,8,9\}\) and \(\D_3 = \{5,6,7,10\}\) satisfy \ref{mainprop1}-\ref{mainprop3}. Then a scheduling sequence, \(\gamma\), is constructed in Steps 16.-18. of Algorithm \ref{algo:sched-policy_design} is as follows:
        \begin{align*}
            \gamma(0) &= v_1 = \pmat{1\\3\\4\\7},\:\:\gamma(1) = v_2 = \pmat{2\\4\\8\\9},\:\:\gamma(2) = v_3 = \pmat{5\\6\\7\\10},\\
            \gamma(3) &= v_1,\:\:\gamma(4) =v_2,\:\:\gamma(5) = v_3,\\
            \vdots.
        \end{align*}
    }
    \end{example}

	{The following result asserts periodicity and stabilizing properties of scheduling sequences obtained from Algorithm \ref{algo:sched-policy_design}.}

    {\begin{proposition}
    \label{prop:algores}
    	Consider an NCS described in \S\ref{s:prob_stat}. Let the matrices \(A_i\), \(B_i\), \(K_i\), \(i=1,2,\ldots,N\), the capacity of the network, \(M\), the probability of data loss, \(p\), and a period, \(\ell\) for a scheduling sequence be given. The following are true:
	\begin{enumerate}[label = \roman*), leftmargin = *]
		\item A scheduling sequence, \(\gamma\), obtained from Algorithm \ref{algo:sched-policy_design} is periodic with period \(\ell\) and ensures ESMS of each plant \(i\in\{1,2,\ldots,N\}\).
		\item A failure message obtained from Algorithm \ref{algo:sched-policy_design} implies that there exists no scheduling sequence of period \(\ell\) that ensures ESMS of each plant \(i\in\{1,2,\ldots,N\}\).
	\end{enumerate}
    \end{proposition}

     \begin{rem}
    \label{rem:compa1}
    \rm{
        Switched systems have appeared before in the scheduling literature, see e.g., \cite{Hristu2001,Lin2005,quevedo2020} and the references therein. In particular, the class of average dwell time switching signals is proven to be a useful tool in the design of stabilizing scheduling sequences for NCSs with continuous-time plants, see e.g., \cite{Lin2005}. A stabilizing average dwell time switching signal involves two conditions on the time interval \(]0:t]\) for every \(t\in\N\): i) an upper bound on the number of switches and ii) a lower bound on the ratio of durations of activation of stable to unstable subsystems \cite{Liberzon_IOSS}. In contrast, our {current} design of stabilizing scheduling sequences for discrete-time NCSs solely relies periodicity of scheduling sequences (and hence, periodicity of the transition probability matrices of the individual plants), and does not involve nor imply restrictions on the behaviour of a scheduling logic on every time interval \(]0:t]\), \(t\in\N\).
    }
    \end{rem}
    \begin{rem}
    \label{rem:compa2}
    \rm{
        In the discrete-time setting, recently in \cite{quevedo2020} we employed switched systems and graph theory to design stabilizing periodic scheduling sequences for NCSs under ideal communication between the plants and their controllers. The design of such sequences involves what is called \(T\)-contractive cycles on the underlying weighted directed graph of an NCS. The results presented in this paper differ from and extend our earlier work \cite{quevedo2020} in the following ways: First, in \cite{quevedo2020} we consider ideal communication scenario between the plants and their controllers, while in this paper we consider communication uncertainties, in particular, probabilistic data losses in the shared communication network. Second, the stability conditions in \cite{quevedo2020} rely on the existence of solutions to a class of feasibility problems involving design of Lyapunov-like functions. The said design problem is numerically complex, and we rely on a partial solution to it in the sense that no solution to their design algorithm does not ensure non-existence of suitable Lyapunov-like functions, see \cite[\S 5]{quevedo2020} for detailed discussions and results. In contrast, in this paper we employ the properties of the matrices corresponding to the modes of operation of the individual plants as the main apparatus for our analysis. Third, the stability conditions presented in \cite{quevedo2020} are only sufficient while in this paper we propose necessary and sufficient conditions for stability under periodic scheduling sequences with a pre-specified period.
    }
    \end{rem}
    \begin{rem}
    \label{rem:periodicity}
    \rm{
        The proposed class of stabilizing scheduling sequences is \emph{static} and thereby easy to implement. Indeed, the sets \(\D_q\), \(q=1,2,\ldots,\ell\) are computed offline and a scheduling sequence is implemented by assigning the elements of \(\D_q\), \(q=1,2,\ldots,\ell\) to \(\gamma\) (in order) and repeating the process eternally.
    }
    \end{rem}
\subsection{Design of static state-feedback controllers}
\label{ss:results_set2}
    {It is evident that for any plant \(i\in\{1,2,\ldots,N\}\), whether the feasibility problem \eqref{e:feasprob} admits a solution or not depends on the matrices \(A_k\), \(k\in R_i\), the probability of data loss, \(p\) and the given period, \(\ell\) of a scheduling sequence. We have so far considered the matrices \(A_i\), \(B_i\) and \(K_i\) to be ``given''. However, a control engineer often has the freedom to design a suitable state-feedback controller matrix, i.e., the matrix \(K_i\), prior to connecting a plant to a shared network. Given the matrices \(A_i,B_i\), \(i=1,2,\ldots,N\) the capacity of the network, \(M\), the probability of data loss, \(p\) and the period, \(\ell\) of a scheduling sequence, we next present an algorithm that designs matrices \(K_i\), \(i=1,2,\ldots,N\) such that the feasibility problem \eqref{e:feasprob} admits a solution for all plants \(i\in\{1,2,\ldots,N\}\).}

    {
    	\begin{algorithm}[htbp]
			\caption{Design of static state-feedback controllers, \(K_i\), \(i=1,2,\ldots,N\)} \label{algo:contrl_design}
		\begin{algorithmic}[1]
			\renewcommand{\algorithmicrequire}{\textbf{Input:}}
			\renewcommand{\algorithmicensure}{\textbf{Output:}}
			
			\REQUIRE The matrices \(A_i\), \(B_i\), \(i=1,2,\ldots,N\), the capacity of the network, \(M\), the probability of data loss, \(p\) and a number \(\ell\in\N\)
			\ENSURE State-feedback controllers, \(K_i\), \(i=1,2,\ldots,N\) or a failure message

			\STATE Construct \(\Gamma = \biggl\{(\D_q)_{q=1}^{\ell}\:|\:\D_q,\:q=1,2,\ldots,\ell\:\text{are distinct and satisfy \ref{mainprop1}-\ref{mainprop2}}\biggr\}\).

             		\FOR {each \((\D_q)_{q=1}^{\ell}\in\Gamma\)}\label{step:repeat-repeat}
				\STATE Set \(\P = \emptyset\).
				\FOR {each \(i=1,2,\ldots,N\)}\label{step:r-repeat}
					\STATE Solve the following feasibility problem for \(P_{i,1}(\iu),\ldots\), \(P_{i,\ell}(\iu)\) and \(P_{i,q}(\is)\in\R^{d\times d}\), where \(i\in\D_q\), \(q\in\{1,2,\ldots,\ell\}\):
					\begin{align}
					\label{e:feasprob1}
						\minimize\:\:&\:\: 1\nonumber\\
                        				\sbjto \:\:&\:\:
                        				\begin{cases}
			&(1-p)A_i^\top P_{i,\tau+1}(\is)A_i + pA_i^\top P_{i,\tau+1}(\iu)A_i - P_{i,\tau}(\iu) \prec 0,\\
			&\hspace*{4cm}\:\:\text{if}\:i\in\D_{\tau+1},\:\tau = 1,2,\ldots,\ell-1,\\
			&A_i^\top P_{i,\tau+1}(\iu)A_i - P_{i,\tau}(\iu) \prec 0,
			\:\:\text{if}\:i\notin\D_{\tau+1},\:\tau = 1,2,\ldots,\ell-1,\\
			&(1-p)A_i^\top P_{i,1}(\is)A_i + pA_i^\top P_{i,1}(\iu)A_i - P_{i,\ell}(\iu) \prec 0,
			\:\:\text{if}\:i\in\D_1,\\
			&A_i^\top P_{i,1}(\iu)A_i - P_{i,\ell}(\iu) \prec 0,
			\:\:\text{if}\:i\notin\D_1,\\
			&P_{i,1}(\iu) = P_{i,1}^\top(\iu),\ldots, P_{i,\ell}(\iu) = P_{i,\ell}^\top(\iu),P_{i,q}(\is) = P_{i,q}^\top(\is),\\
			&P_{i,1}(\iu) \succ 0,\ldots, P_{i,\ell}(\iu) \succ 0,P_{i,q}(\is) \succ 0.
						\end{cases}
					\end{align}
					\IF {there exists a solution to \eqref{e:feasprob1}}
						\STATE Solve the following feasibility problem for \(q\in\{1,2,\ldots,\ell\}\) and \(Y_i\in\R^{m\times d}\):
						\begin{align}
						\label{e:feasprob2}
						\minimize\:\:&\:\: 1\nonumber\\
                        				\sbjto \:\:&\:\:
                        				\begin{cases}
							&\bigl(A_i P_{i,q}^{-1}(\is)+B_i Y_i\bigr)^\top P_{i,q+1}(\iu) \bigl(A_i P_{i,q}^{-1}(\is) + B_i Y_i\bigr) - P_{i,q}^{-1}(\is) \prec 0,\\
							&\hspace*{4cm}\:\:\text{if}\:i\in\D_q\:\text{and}\:i\notin\D_{q+1},\\
			&\biggl(A_i\bigl((1-p)P_{i,q+1}(\is)+pP_{i,q+1}(\iu)\bigr)^{-1}+B_iY_i\biggr)^\top P_{i,q+1}(\iu)\times\\
			&\:\:\biggl(A_i\bigl((1-p)P_{i,q+1}(\is)+pP_{i,q+1}(\iu)\bigr)^{-1}+B_iY_i\biggr)
			- P_{i,q}^{-1}(\is) \prec 0,\\
			&\hspace*{4cm}\:\:\text{if}\:i\in\D_q\:\text{and}\:i\in\D_{q+1},\\
			&P_{i,\ell+1} := P_{i,1}.
						\end{cases}
						\end{align}
						
						\IF {there exists a solution to \eqref{e:feasprob2}}
							\STATE Set
							\(K_i = Y_i P_{i,q}(\is)\) and
							solve the following feasibility problem for \(P_{i,1}(\is),\ldots,P_{i,q-1}(\is)\), \(P_{i,q+1}(\is),\ldots,P_{i,\ell}(\is)\in\R^{d\times d}\), where \(i\in\D_q\), \(q\in\{1,2,\ldots,\ell\}\):
						\begin{align}
						\label{e:feasprob3}
						\minimize\:\:&\:\: 1\nonumber\\
                        				\sbjto \:\:&\:\:
                        				\begin{cases}
							&(1-p)(A_i+B_iK_i)^\top P_{i,\tau+1}(\is)(A_i+B_iK_i)
							+p(A_i+B_iK_i)^\top P_{i,\tau+1}(\iu)(A_i+B_iK_i)\\
							&\hspace*{3cm}-P_{i,\tau}(\is)\prec 0,
			\:\:\text{if}\:i\in\D_{\tau+1},\:\tau=1,2,\ldots,\ell-1,\\
			&(A_i+B_iK_i)^\top P_{i,\tau+1}(\iu)(A_i+B_iK_i) - P_{i,\tau}(\is) \prec 0,\\
			&\hspace*{3cm}\:\:\text{if}\:i\notin\D_{\tau+1},\:\tau=1,2,\ldots,\ell-1,\\
			&(1-p)(A_i+B_iK_i)^\top P_{i,1}(\is)(A_i+B_iK_i)+p(A_i+B_iK_i)^\top P_{i,1}(\iu)(A_i+B_iK_i)\\
			&\hspace*{4cm}-P_{i,\ell}(\is) \prec 0,
			\:\:\text{if}\:i\in\D_1,\\
			&(A_i+B_iK_i)^\top P_{i,1}(\iu)(A_i+B_iK_i) - P_{i,\ell}(\is)\prec 0,
			\:\:\text{if}\:i\notin\D_1,\\
			&P_{i,1}(\is)=P_{i,1}^\top(\is),\ldots,P_{i,q-1}(\is)=P_{i,q-1}^\top(\is),\\
			&P_{i,q+1}(\is)=P_{i,q+1}^\top(\is),\ldots,P_{i,\ell}(\is)= P_{i,\ell}^\top(\is),\\
			&P_{i,1}(\is)\succ 0, \ldots, P_{i,q-1}(\is)\succ 0,\\
			&P_{i,q+1}(\is)\succ 0, \ldots, P_{i,\ell}(\is)\succ 0.
						\end{cases}
						\end{align}
						\IF {there exists a solution to \eqref{e:feasprob3}}
							\STATE Set \(\P = \P\cup\{i\}\).
						\ENDIF
						\ENDIF
					\ENDIF
				\ENDFOR
				\IF {\(\P = \{1,2,\ldots,N\}\)}
					\STATE Output \(K_i\), \(i=1,2,\ldots,N\) and exit.
				\ENDIF
			\ENDFOR
             		\STATE Output ``FAIL''.
			\end{algorithmic}
	\end{algorithm}
    }

     {Algorithm \ref{algo:contrl_design} relies on the matrices \(A_i\), \(B_i\), \(i=1,2,\ldots,N\), the capacity of the network, \(M\), the probability of data loss, \(p\) and the given period, \(\ell\) of a scheduling sequence to design suitable state-feedback controller matrices, \(K_i\), \(i=1,2,\ldots,N\). While \(A_i\), \(B_i\), \(i=1,2,\ldots,N\), \(p\) and \(\ell\) appear in the feasibility problems \eqref{e:feasprob1}-\eqref{e:feasprob3}, \(M\) determines the number of elements in each \(\D_q\), \(q = 1,2,\ldots,\ell\) that need to satisfy the required conditions. If there does not exist \((\D_q)_{q=1}^{\ell}\in\Gamma\) such that solutions to the feasibility problems \eqref{e:feasprob1}-\eqref{e:feasprob3} are obtained for all \(i\in\{1,2,\ldots,N\}\), then Algorithm \ref{algo:contrl_design} reports a failure.}

    {
    \begin{proposition}
    \label{prop:mainres}
    	Consider an NCS described in \S\ref{s:prob_stat}. Let the matrices \(A_i\), \(B_i\), \(i=1,2,\ldots,N\), the capacity of the network, \(M\), the probability of data loss, \(p\) and a period, \(\ell\) of a scheduling sequence be given. Suppose that matrices \(K_i\), \(i=1,2,\ldots,N\) are obtained from Algorithm \ref{algo:contrl_design}. Then there exists \((\D_q)_{q=1}^{\ell}\) satisfying \ref{mainprop1}-\ref{mainprop2} such that the feasibility problem \eqref{e:feasprob} admits a solution \(P_{i,1}(j),\ldots\), \(P_{i,\ell}(j)\), \(j\in R_i\) for each \(i\in\{1,2,\ldots,N\}\).
    \end{proposition}
    }

    {Notice that Proposition \ref{prop:mainres} relies on sufficient conditions for the existence of state-feedback controller matrices that are favourable for our class of stabilizing scheduling sequences. Indeed, our computation of \(K_i\), \(i=1,2,\ldots,N\), employed in Algorithm \ref{algo:contrl_design}, is not unique. Consequently, a failure message obtained from Algorithm \ref{algo:contrl_design} does not imply the non-existence of state-feedback controllers such that the plants and the shared communication network under consideration together admit a purely time-dependent stabilizing periodic scheduling sequence. }


    {
    \begin{rem}
    \label{rem:compa3}
    \rm{
    	The design of static state-feedback controllers for switched systems whose switching signals are time-homogeneous Markov chains was addressed earlier in the literature, see e.g., the works \cite{Zhang2009,Kordonis2014}.
    Our design of state-feedback controllers in Algorithm \ref{algo:contrl_design} caters to the existence of solutions to the feasibility problem \eqref{e:feasprob}. We exploit algebraic properties of the target inequalities (see our proof of Proposition \ref{prop:mainres}) for this purpose. While our analysis tool is similar in spirit to \cite{Zhang2009}, we tackle a more general setting compared to  \cite{Zhang2009} in the following sense: first, we deal with switched systems whose switching signals are time-inhomogeneous, and second, our design caters to simultaneous stability of \(N\) such systems.
    }
    \end{rem}
    }

    We now present a set of examples to demonstrate our techniques.
\section{Numerical experiments}
\label{s:num_ex}
	\begin{experiment}
    \label{ex:num_ex1}
    \rm{
        We consider an NCS with \(N=2\), where
        \begin{align*}
            A_1 &= \pmat{0.65 & 0.2\\-0.1 & 1.1},\:\:B_1 = \pmat{0\\1},\:\:K_1 = \pmat{0.1 & -1.1}\\
            \intertext{and}
            A_2 &= \pmat{0.7 & 0.1\\-0.2 & 1.1},\:\:B_2 = \pmat{0\\1},\:\:K_2 = \pmat{0.2 & -1.1}.
        \end{align*}
         The plant and controller dynamics are borrowed from \cite[\S IVA]{Liu2019}. Let the network capacity, \(M = 1\), the data loss probability, \(p = 0.5\) and the desired period of a scheduling sequence, \(\ell = 2\). We employ Algorithm \ref{algo:sched-policy_design} to solve Problem \ref{prob:main} in the above setting. The following steps are executed:
        \begin{enumerate}[label = Step \Roman*., leftmargin = *]
            \item Compute \(\ell_{\min} = 2\) and note that \(\ell = \ell_{\min}\).
            \item Construct \(\Gamma = \Biggl\{\biggl(\{1\},\{2\}\biggr),\biggl(\{2\},\{1\}\biggr)\Biggr\}\).
            \item We solve the feasibility problem \eqref{e:feasprob}. For \(\D_1 = \{2\}\) and \(\D_2 = \{1\}\), the following values of \(P_{i,1}(j)\), \(P_{i,2}(j)\), \(j\in R_i\), \(i=1,2\) are obtained:
                \begin{align*}
                    P_{1,1}(1_s) &= \pmat{891.90358 & 74.12886\\74.12886 & 673.79367},\:\:&\:\:P_{1,1}(1_u) &= \pmat{749.2162 & -253.33635\\-253.33635 & 2245.0484}\\
                    P_{1,2}(1_s) &= \pmat{797.2495 & -5.9026364\\-5.9026364 & 394.41295},\:\:&\:\:P_{1,2}(1_u) &= \pmat{815.56198 & -375.29485\\-375.29485 & 2929.1336},\\
                    P_{2,1}(2_s) &= \pmat{1116.1217 & -11.624074\\-11.624074 & 294.60972},\:\:&\:\:P_{2,1}(2_u) &= \pmat{1241.0856 & -537.21264\\-537.21264 & 2134.3708},\\
                    P_{2,2}(2_s) &= \pmat{1225.6192 & 61.169859\\61.169859 & 806.16873},\:\:&\:\:P_{2,2}(2_u) &= \pmat{1140.0419 & -378.54181\\-378.54181 & 1626.5343}.
                \end{align*}
                It follows that the above \(P_{i,1}(j)\), \(P_{i,2}(j)\), \(j\in R_i\), \(i=1,2\) are symmetric and positive definite and satisfy conditions \eqref{e:condn1}-\eqref{e:condn4} with positive definite \(Q_{i,1}(j)\), \(Q_{i,2}(j)\), \(j\in R_i\), \(i = 1,2\) given below:
                \begin{align*}
                     Q_{1,1}(1_s) &= \pmat{551.19716 & -30.703886\\-30.703886 & 641.53744},\:\:&\:\: Q_{1,1}(1_u) &= \pmat{367.11421 & -42.907906\\-42.907906 & 285.90995},\\
                     Q_{1,2}(1_s) &= \pmat{480.70565 & -103.30074\\-103.30074 & 364.44431},\:\:&\:\:
                     Q_{1,2}(1_u) &= \pmat{443.63392 & -49.668868\\-49.668868 & 294.12439},\\
                    Q_{2,1}(2_s) &= \pmat{557.50121 & -91.427005\\-91.427005 & 283.2093},\:\:&\:\:Q_{2,1}(2_u) &= \pmat{511.41201 & 24.728926\\24.728326 & 238.14305},\\
                    Q_{2,2}(2_s) &= \pmat{648.10341 & -21.332397\\-21.332397 & 794.3827},\:\:&\:\:Q_{2,2}(2_u) &= \pmat{437.10932 & 11.95756\\11.95756 & 205.5871}.
                \end{align*}
            \item A scheduling sequence, \(\gamma\), is constructed as
                \begin{align*}
                    \gamma(0) &= 1,\:\:\gamma(1) = 2,\\
                    \gamma(2) &= 1,\:\:\gamma(3) = 2,\\
                    \vdots.
                \end{align*}
            \item For each plant \(i\in\{1,2\}\), we pick \(100\) different initial conditions \(x_i(0)\) from the interval \([-1,+1]^{2}\) and plot \(\biggl(\EE\{\norm{x_{i}(t)}^{2}\}\biggr)_{t\in\N_0}\), see Figures \ref{fig:x_plot1} and \ref{fig:x_plot2}. ESMS is demonstrated for each plant in the NCS under consideration.
        \end{enumerate}
    }
    \begin{figure}[htbp]
        \includegraphics[scale = 0.4]{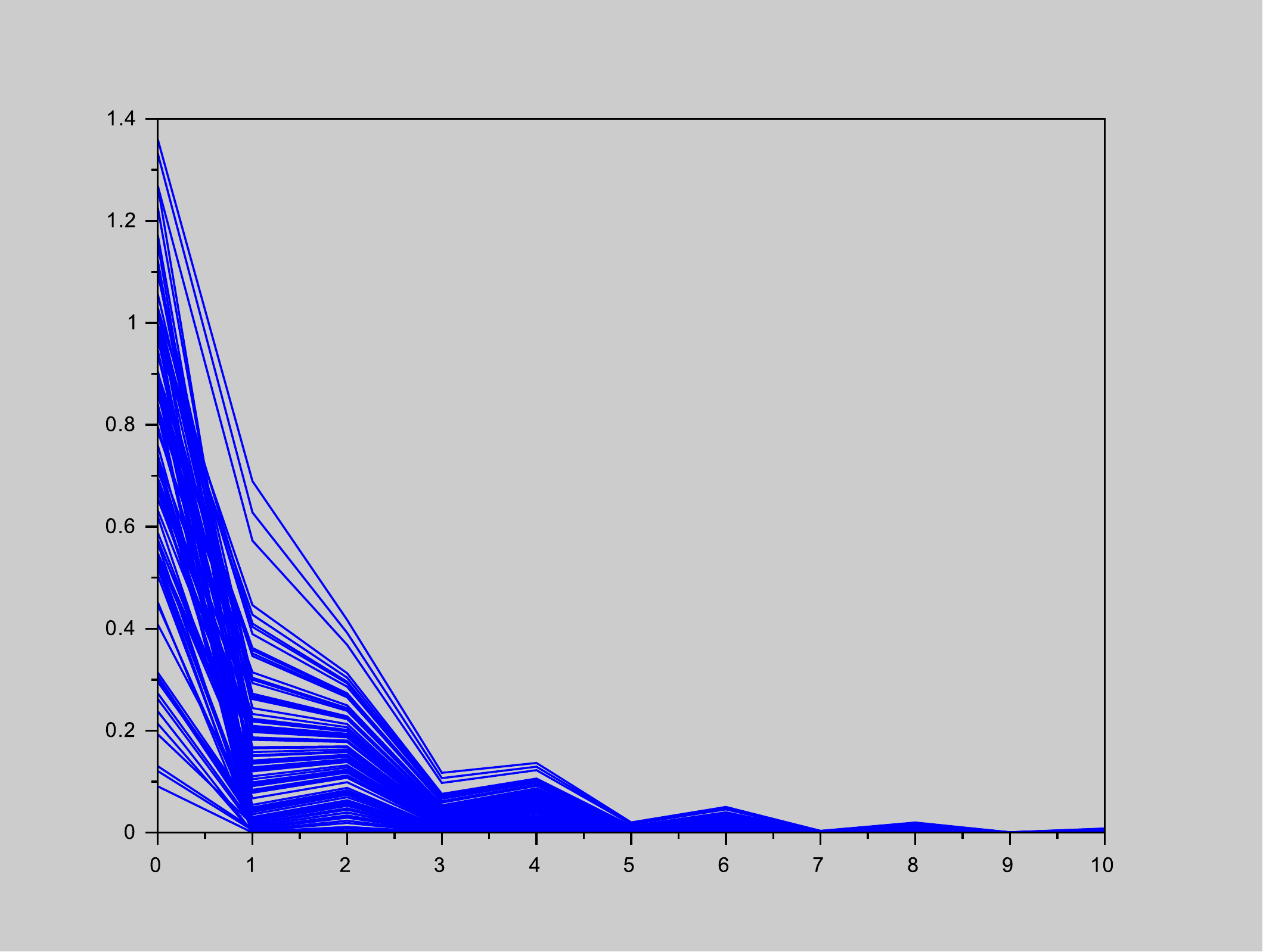}
        \caption{\(\EE\biggl\{\norm{x_{1}(t)}^{2}\biggr\}\) versus \(t\) for Example \ref{ex:num_ex1}}\label{fig:x_plot1}
    \end{figure}
    \begin{figure}[htbp]
        \includegraphics[scale = 0.4]{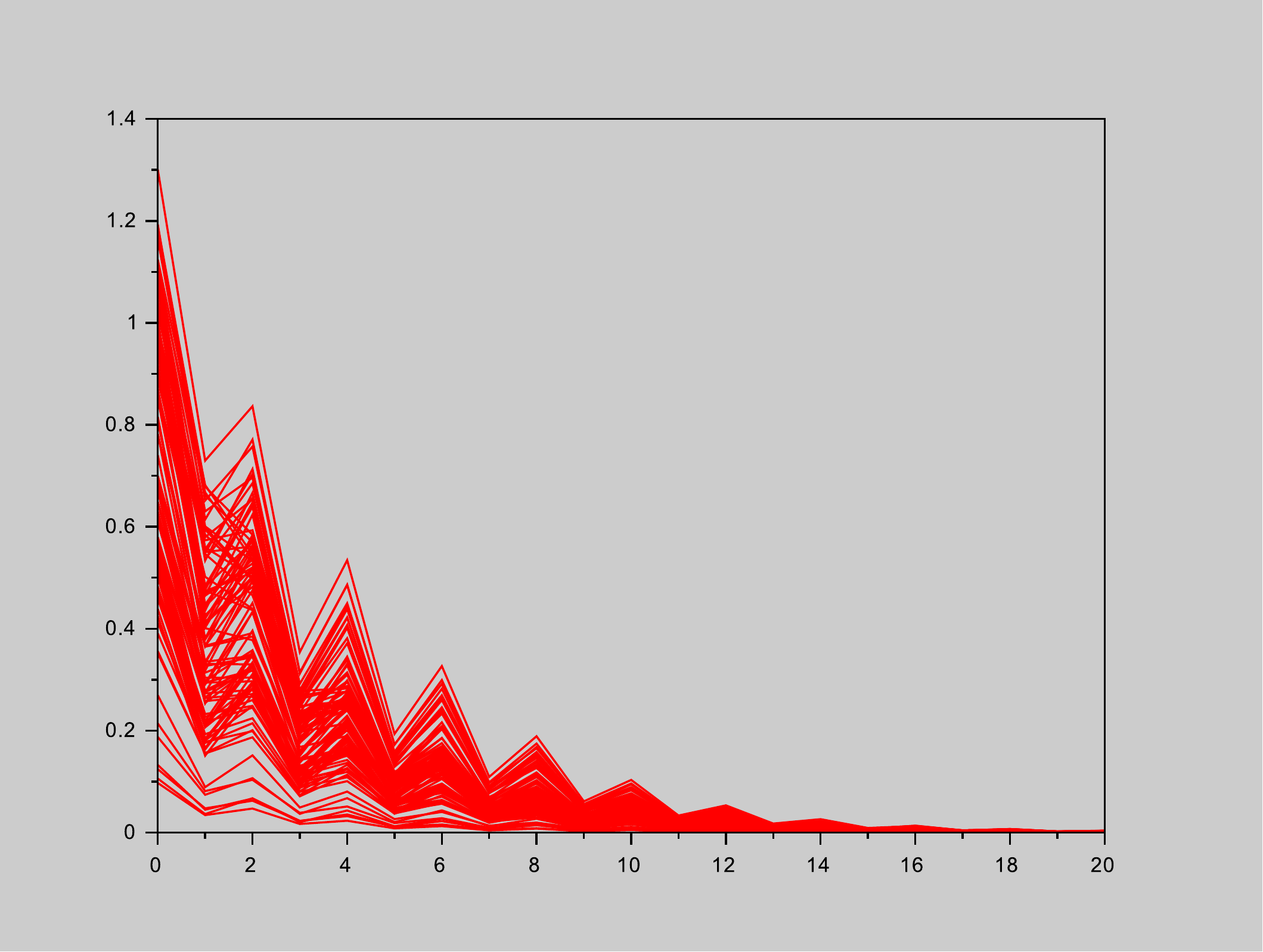}
        \caption{\(\EE\biggl\{\norm{x_{2}(t)}^{2}\biggr\}\) versus \(t\) for Example \ref{ex:num_ex1}}\label{fig:x_plot2}
    \end{figure}
    \end{experiment}

    \begin{experiment}
    \label{ex:num_ex2}
    \rm{
    {
    		Consider the setting of Experiment \ref{ex:num_ex1}. Suppose that the controller matrices, \(K_i\), \(i=1,2\) were not known. We apply Algorithm \ref{algo:contrl_design} to design \(K_i\), \(i=1,2\) based on the information of the matrices \(A_i\), \(B_i\), \(i=1,2\), the capacity of the network, \(M\), the probability of data loss, \(p\) and the given period, \(\ell\) of a scheduling sequence, such that there exists \((\D_q)_{q=1}^{\ell}\) satisfying \ref{mainprop1}-\ref{mainprop2} for which the feasibility problem \eqref{e:feasprob} admits solutions \(P_{i,1}(j),\ldots\), \(P_{i,\ell}(j)\), \(j\in R_i\) for each \(i\in\{1,2\}\). The following steps are carried out:
		\begin{enumerate}[label = Step \Roman*., leftmargin = *]
			\item Fix \(\D_1 = \{2,3\}\) and \(\D_2 = \{1,3\}\). We have \((\D)_{q=1}^{2}\in\Gamma\).
			\item Fix \(i=1\).
			\begin{enumerate}[label = \roman*), leftmargin = *]
				\item There exist symmetric and positive definite matrices \(P_{1,1}(1_u)\), \(P_{1,2}(1_u)\), \(P_{1,2}(1_s)\), described in Experiment \ref{ex:num_ex1}, that solve the feasibility problem \eqref{e:feasprob1}. Indeed,
				\begin{align*}
					(1-p)A_1^\top P_{1,2}(1_s)A_1 + p A_1^\top P_{1,2}(1_u)A_1 - P_{1,1}(1_u) &= -Q_{1,1}(1_u)\prec 0,\\
					A_1^\top P_{1,1}(1_u)A_1 - P_{1,2}(1_u) &= -Q_{1,2}(1_u)\prec 0,
				\end{align*}
				where \(Q_{1,1}(1_u)\) and \(Q_{1,2}(1_u)\) are as described in Experiment \ref{ex:num_ex1}.
				\item There exist \(q = 2\) and \(Y_1 = \pmat{0.0001048 & -0.0027874}\) that solve the feasibility problem \eqref{e:feasprob2}. Indeed,
				\begin{align*}
					&(A_1 P_{1,2}^{-1}(1_s) + B_1Y_1)^\top P_{1,1}(1_u) (A_1 P_{1,2}^{-1}(1_s) + B_1Y_1) - P_{1,2}^{-1}(1_s)\\
					=&\pmat{-0.0007517 & 0.0003\\0.0003 & -0.0023336} \prec 0.
				\end{align*}
				\item Set \(K_1 = Y_1 P_{1,2}(1_s) = \pmat{0.1 & -1.1}\).
				\item The above controller matrix is already demonstrated to be favourable. Indeed, we have that there exists a symmetric and positive definite matrix \(P_{1,1}(1_s)\), as described in Experiment \ref{ex:num_ex1}, that solves the feasibility problem \eqref{e:feasprob3}. Indeed,
				\begin{align*}
					&\hspace*{2cm}(1-p)(A_1+B_1K_1)^\top P_{1,2}(1_s)(A_1+B_1K_1)+p(A_1+B_1K_1)^\top P_{1,2}(1_u)(A_1+B_1K_1)\\
					&\hspace*{8cm}\quad\quad- P_{1,1}(1_s) = -Q_{1,1}(1_s)\prec 0,\\
					&\hspace*{2cm}(A_1+B_1K_1)^\top P_{1,1}(1_u)(A_1+B_1K_1)-P_{1,2}(1_s) = -Q_{1,2}(1_s) \prec 0,
				\end{align*}
				where \(Q_{1,1}(1_s)\) and \(Q_{1,2}(1_s)\) are as described in Experiment \ref{ex:num_ex1}.
			\end{enumerate}
			\item Fix \(i=2\).
			\begin{enumerate}[label = \roman*), leftmargin = *]
				\item There exist symmetric and positive definite matrices \(P_{2,1}(2_u)\), \(P_{2,2}(2_u)\), \(P_{2,1}(2_s)\), described in Experiment \ref{ex:num_ex1}, that solve the feasibility problem \eqref{e:feasprob1}. Indeed,
				\begin{align*}
					(1-p)A_2^\top P_{2,1}(2_s)A_2 + p A_2^\top P_{2,1}(2_u)A_2 - P_{2,2}(2_u) &= -Q_{2,2}(2_u)\prec 0,\\
					A_2^\top P_{2,2}(2_u)A_2 - P_{2,1}(2_u) &= -Q_{2,1}(2_u)\prec 0,
				\end{align*}
				where \(Q_{2,1}(2_u)\) and \(Q_{2,2}(2_u)\) are as described in Experiment \ref{ex:num_ex1}.
				\item There exist \(q = 1\) and \(Y_2 = \pmat{0.0001404 & -0.0037282}\) that solve the feasibility problem \eqref{e:feasprob2}. Indeed,
				\begin{align*}
					&(A_2 P_{2,1}^{-1}(2_s) + B_2Y_2)^\top P_{2,2}(2_u) (A_2 P_{2,1}^{-1}(2_s) + B_2Y_2) - P_{2,1}^{-1}(2_s)\\
					=&\pmat{-0.0004425 & 0.0002267\\0.0002267 & -0.0032444} \prec 0.
				\end{align*}
				\item Set \(K_2 = Y_2 P_{2,1}(2_s) = \pmat{0.2 & -1.1}\).
				\item The above controller matrix is already demonstrated to be favourable. Indeed, we have that there exists a symmetric and positive definite matrix \(P_{2,2}(2_s)\), as described in Experiment \ref{ex:num_ex1}, that solves the feasibility problem \eqref{e:feasprob3}. Indeed,
				\begin{align*}
					&\hspace*{2cm}(1-p)(A_2+B_2K_2)^\top P_{2,1}(2_s)(A_2+B_2K_2)+p(A_2+B_2K_2)^\top P_{2,1}(2_u)(A_2+B_2K_2)\\
					&\hspace*{8cm}\quad\quad- P_{2,2}(2_s) = -Q_{2,2}(2_s)\prec 0,\\
					&\hspace*{2cm}(A_2+B_2K_2)^\top P_{2,2}(2_u)(A_2+B_2K_2)-P_{2,1}(2_s) = -Q_{2,1}(2_s) \prec 0,
				\end{align*}
				where \(Q_{2,1}(2_s)\) and \(Q_{2,2}(2_s)\) are as described in Experiment \ref{ex:num_ex1}.
			\end{enumerate}
		\end{enumerate}
	}
	}
	\end{experiment}
	
	\begin{experiment}
	\label{ex:num_ex4}
	{\rm{
		We consider an NCS with \(N = 3\), where each plant is a batch reactor. We employ a discretised version of a linearised batch reactor model presented in \cite[\S IVA]{Walsh2002}. In particular, we have
		\(A_i = \pmat{1.0795 & -0.0045 & 0.2896 & -0.2367\\-0.0272 & 0.8101 & -0.0032 & 0.0323\\0.0447 & 0.1886 & 0.7317 & 0.2354\\0.0010 & 0.1888 & 0.0545 & 0.9115}\) and \(B_i = \pmat{0.0006 & -0.0239\\0.2567 & 0.0002\\0.0837 & -0.1346\\0.0837 & -0.0046}\), \(i=1,2,3\). Let the capacity of the network \(M = 2\), the probability of data loss, \(p=0.2\) and the desired period of a scheduling sequence, \(\ell = \ell_{\min} = 2\).
		
		We first apply Algorithm \ref{algo:contrl_design} to design state-feedback controller matrices, \(K_i\), \(i=1,2,3\) such that there exists \((\D_q)_{q=1}^{\ell}\) satisfying \ref{mainprop1}-\ref{mainprop2} with the feasibility problem \eqref{e:feasprob} admitting a solution \(P_{i,1}(j),P_{i,2}(j)\), \(j\in R_i\) for each \(i=1,2,3\). 
		We then generate a scheduling sequence, \(\gamma\), as follows:
                \begin{align*}
                    \gamma(0) &= \pmat{2\\3},\:\:\gamma(1) = \pmat{1\\3},\\
                    \gamma(2) &= \pmat{2\\3},\:\:\gamma(3) = \pmat{1\\3},\\
                    \vdots.
                \end{align*}
            For each plant \(i\in\{1,2,3\}\), we pick \(100\) different initial conditions \(x_i(0)\) from the interval \([-1,+1]^{2}\) and plot \(\biggl(\EE\{\norm{x_{i}(t)}^{2}\}\biggr)_{t\in\N_0}\), see Figures \ref{fig:x_plot3} and \ref{fig:x_plot5}. Exponential second moment stability is demonstrated for each plant in the NCS under consideration.
	}
	}
\begin{figure}[htbp]
        \includegraphics[scale = 0.4]{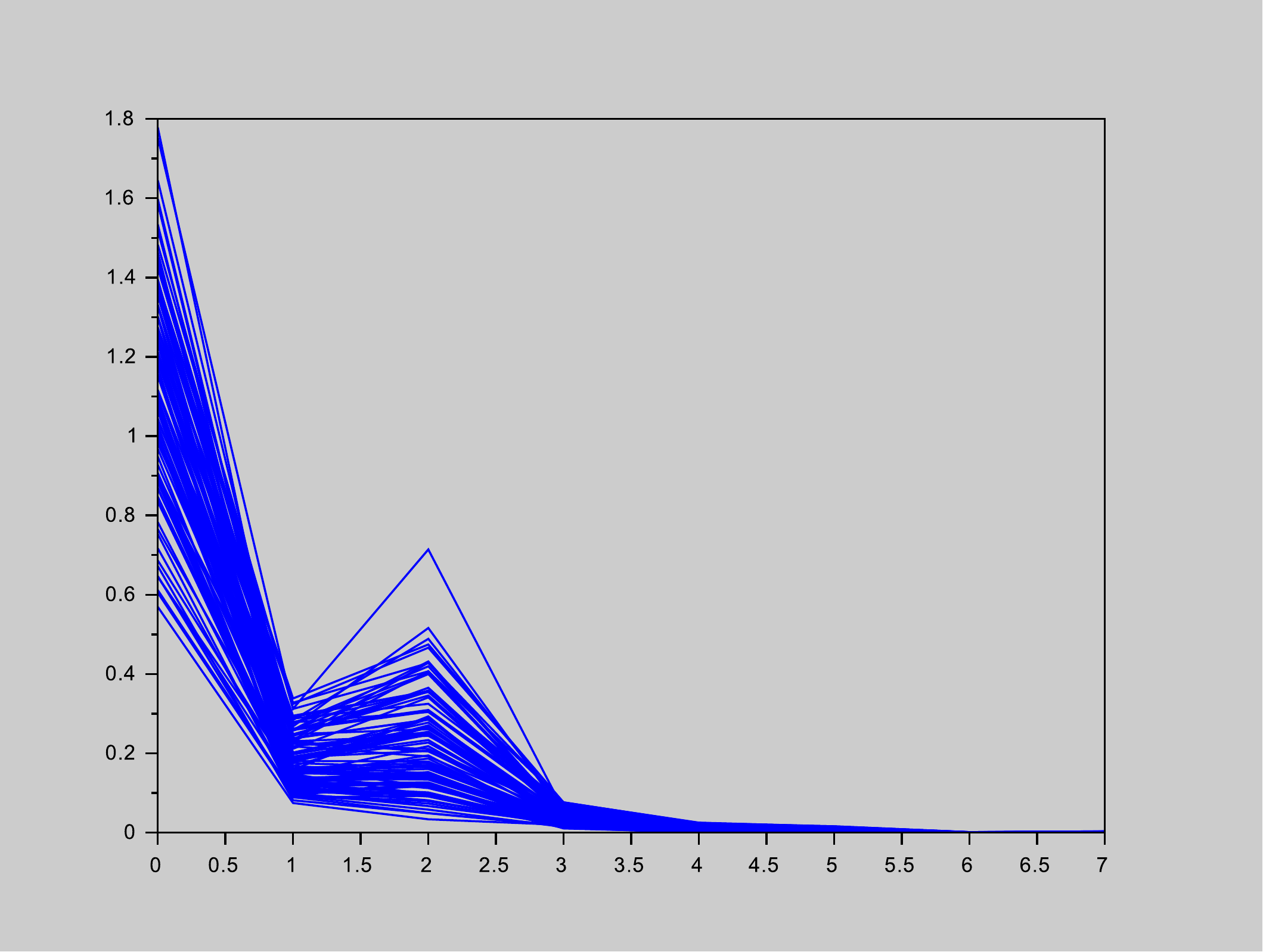}
        \caption{\(\EE\biggl\{\norm{x_{1}(t)}^{2}\biggr\}\) versus \(t\) for Example \ref{ex:num_ex4}}\label{fig:x_plot3}
    \end{figure}
    \begin{figure}[htbp]
        \includegraphics[scale = 0.4]{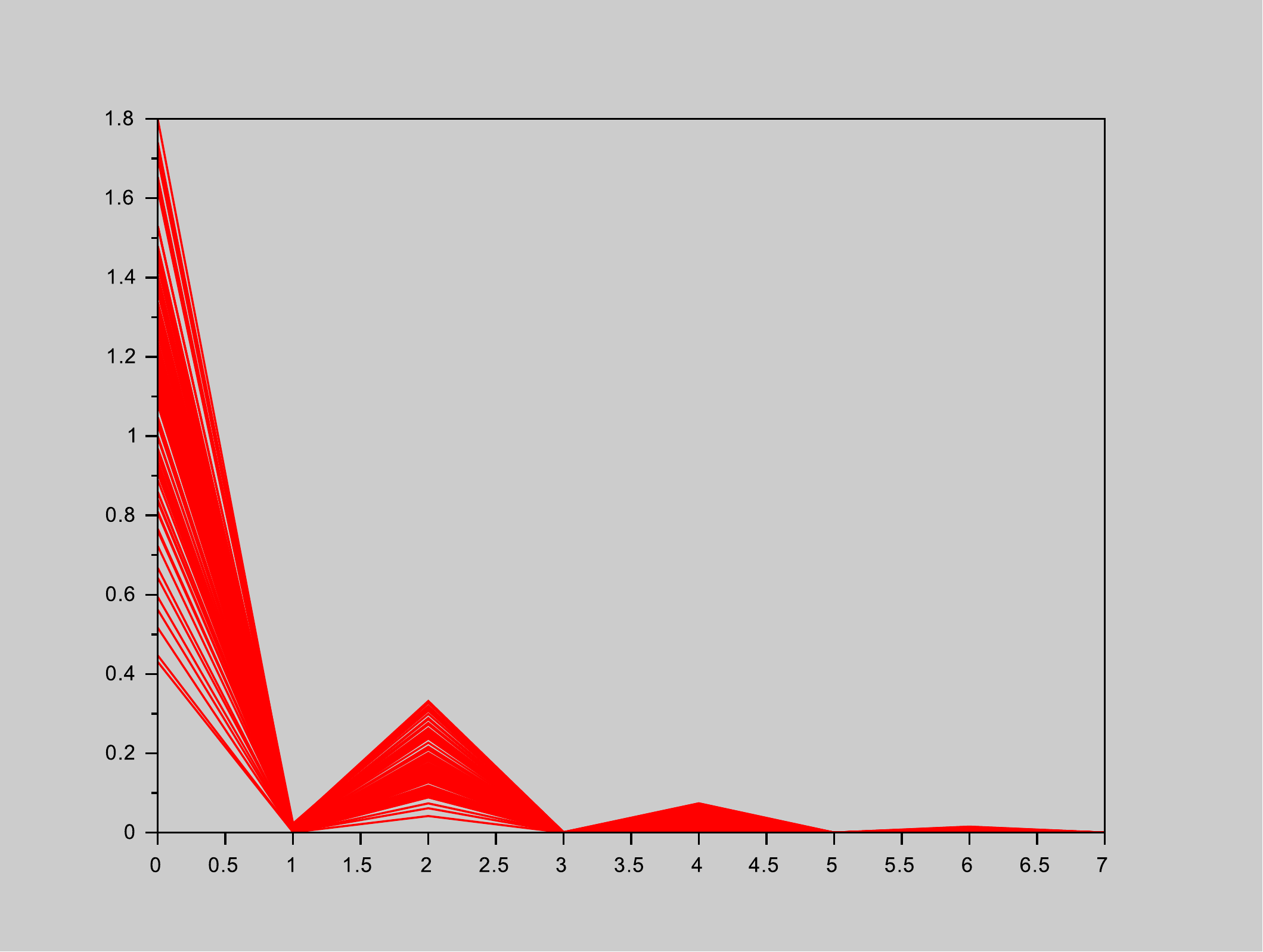}
        \caption{\(\EE\biggl\{\norm{x_{2}(t)}^{2}\biggr\}\) versus \(t\) for Example \ref{ex:num_ex4}}\label{fig:x_plot4}
    \end{figure}
    \begin{figure}[htbp]
        \includegraphics[scale = 0.4]{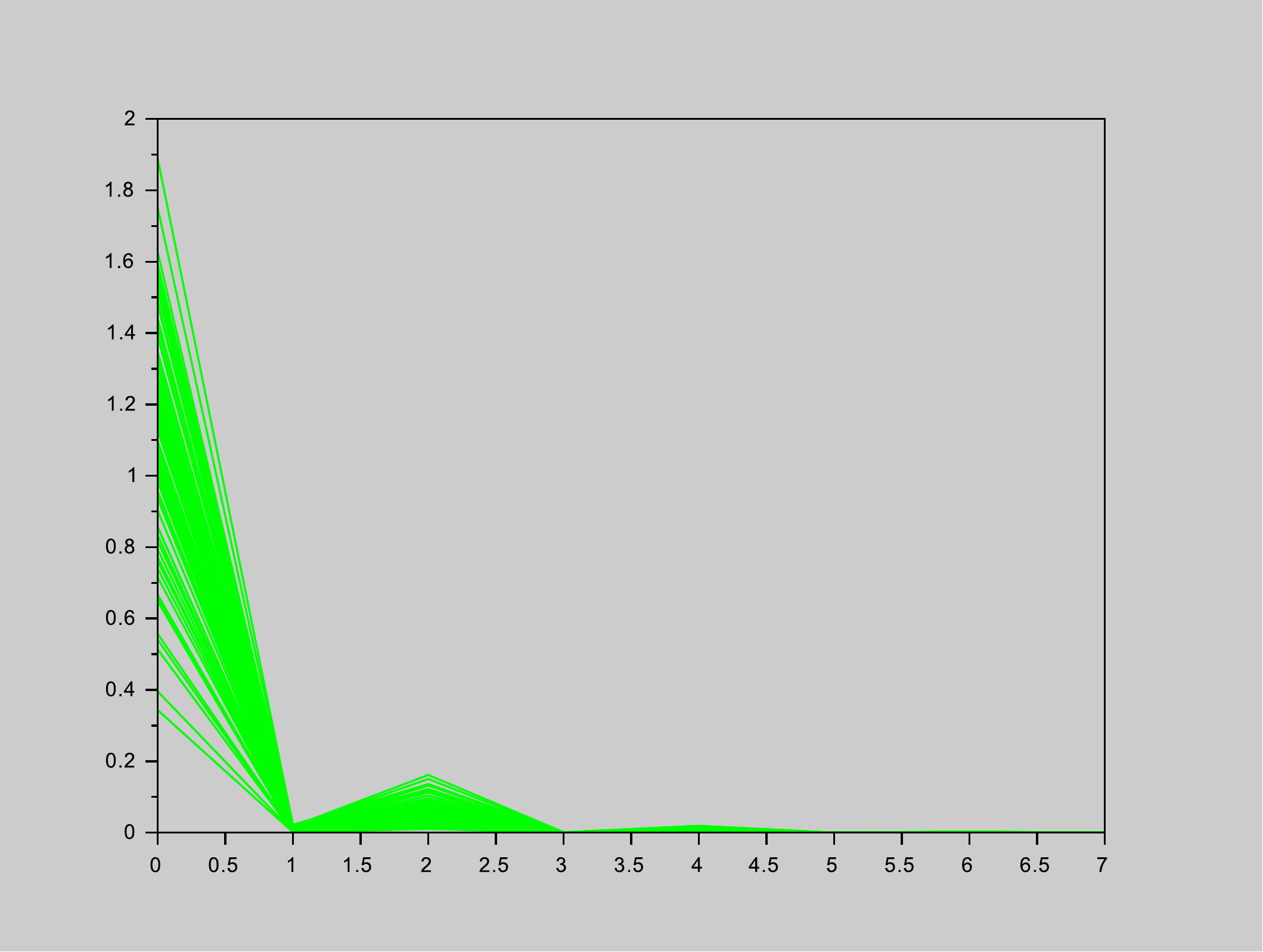}
        \caption{\(\EE\biggl\{\norm{x_{3}(t)}^{2}\biggr\}\) versus \(t\) for Example \ref{ex:num_ex4}}\label{fig:x_plot5}
    \end{figure}
	\end{experiment}
	
	\begin{experiment}
	\label{ex:num_ex3}
	\rm{
		{A key task of Algorithm \ref{algo:sched-policy_design} is solving the feasibility problem \eqref{e:feasprob}. In this experiment we test the worst-case computation time required for solving \eqref{e:feasprob} for all \(i\in\D_q\), \(q=1,2,\ldots,\ell\) in a large-scale setting, i.e., NCSs with a large number of plants. We generate unstable matrices \(A_i\in\R^{5\times 5}\) and vectors \(B_i\in\R^{5\times 1}\) with entries from the interval \([-2,2]\) and the set \(\{0,1\}\), respectively, chosen uniformly at random and ensuring that each pair of matrices \((A_i,B_i)\), \(i=1,2,\ldots,N\) is controllable. The controller matrices, \(K_i\in\R^{1\times 5}\) are computed as a linear quadratic regulator of plant \(i\) with state cost = \(5I_{5\times 5}\) and input cost = 1. We fix \(M = 10\), \(p = 0.5\) and \(\ell = \ell_{\min}\). For all possible choices of \((\D_q)_{q=1}^{\ell}\in\Gamma\), we solve the feasibility problem \eqref{e:feasprob} for all \(i\in\D_q\), \(q=1,2,\ldots,\ell\).
		
		The above procedure is carried out for various large values of \(N\) by employing the LMI solver toolbox in MATLAB R2020a on an Intel 17-8550U, 8 GB RAM, 1 TB HDD PC with Windows 10 operating system. The collected data are summarized in Table \ref{tab:execution_time}. Not surprisingly, we observe that the computation time under consideration increases as the number of plants in an NCS increases. However, since the sets \(\D_q\), \(q=1,2,\ldots,\ell\) are computed offline, a large computation time does not affect the performance of our algorithm. We implement a scheduling sequence, \(\gamma\), by assigning the elements of \(\D_q\), \(q=1,2,\ldots,\ell\) to \(\gamma\) (in order) and repeating the process eternally.}
	}
	\end{experiment}
	\begin{table}[htbp]
	\centering
	\begin{tabular}{|c | c | c|c|c|c|}
		\hline
		\(N\) & \(M\) & \(d\) & \(p\) & \(\ell\) & Time to solve \eqref{e:feasprob} for all \(i\in\D_q\), \(q=1,2,\ldots,\ell\) (in sec) \\
        \hline
        \(100\) & \(10\) & \(5\) & \(0.5\) & \(10\) & 81.5625\\
        \hline
        \(250\) & \(10\) & \(5\) & \(0.5\) & \(25\) & 1274.4141\\
        \hline
        \(500\) & \(10\) & \(5\) & \(0.5\) & \(50\) & 10195.312\\
        \hline
        \(1000\) & \(10\) & \(5\) & \(0.5\) & \(100\) & 81562.5\\
       	\hline
	\end{tabular}
    \vspace*{0.2cm}
	\caption{Data for Experiment \ref{ex:num_ex3}}\label{tab:execution_time}
	\end{table}
\section{Concluding remarks}
\label{s:concln}
    {In this paper we considered NCSs whose communication networks have limited bandwidth and are prone to data losses. We designed stabilizing purely time-dependent periodic scheduling sequences for NCSs  We relied on the existence of subsets of all plants that satisfy certain conditions for this purpose. The proposed stability conditions are necessary and sufficient. We also presented an algorithm to design state-feedback controllers such that the plants and the shared communication network in an NCS together admit the proposed class of stabilizing scheduling sequences. A natural extension of our work is to accommodate other forms of network induced uncertainties such as access delays, quantization errors, etc. in the feedback control loop. This topic is currently under investigation and will be reported elsewhere.}
\section{Proofs of our results}
\label{s:proofs}
\begin{proof}[Proof of Lemma \ref{lem:auxres1}]
        We have
        \begin{align*}
            \Pi_i(t) &=
            \begin{cases}
                \pmat{1-p & p\\1-p & p},&\:\:\text{if \(i\) is an element of \(\gamma(t)\)},\\
                \pmat{0 & 1\\0 & 1},&\:\:\text{if \(i\) is not an element of \(\gamma(t)\)}.
            \end{cases}
        \end{align*}
        In both the cases, the elements of \(\Pi_i(t)\) are non-negative and the elements of each row sum up to \(1\). In addition,
        \begin{align*}
            \Phi_{i0} &=
            \begin{cases}
                \pmat{1-p & p},&\:\:\text{if \(i\) is an element of \(\gamma(0)\)},\\
                \pmat{0 & 1},&\:\:\text{if \(i\) is not an element of \(\gamma(0)\)}.
            \end{cases}
        \end{align*}
        In both the cases, the elements of \(\Phi_{i0}\) are non-negative and their sum is \(1\).

        The assertion of Lemma \ref{lem:auxres1} follows at once.
    \end{proof}

    \begin{proof}[Proof of Lemma \ref{lem:auxres2}]
        Follows from the observation that the conditions \(\gamma(t) = \gamma(t+\ell)\) for all \(t\in\N_0\) and \(\Pi_i(t) = \Pi_i(t+\ell)\) for all \(t\in\N\) is equivalent to the fact that the network access status of any plant \(i\) at time \(t\) is the same as its network access status at time \(t+\ell\) for all \(t\in\N_0\).
    \end{proof}

    \begin{proof}[Proof of Theorem \ref{t:mainres}]
        (Sufficiency) Suppose that there exist distinct sets \(\D_q\subset\{1,2,\ldots,N\}\), \(q = 1,2,\ldots,\ell\), that satisfy conditions \ref{mainprop1}-\ref{mainprop3}. Let \(v_q\in\Svec\) be the vector containing the elements of the set \(\D_q\), \(q=1,2,\ldots,\ell\). Consider a scheduling sequence, \(\gamma\), that obeys
        \begin{align*}
            \gamma(t) &= v_1,\:\:t=0,\ell,2\ell,3\ell,\ldots,\\
            \gamma(t) &= v_2,\:\:t=1,\ell+1,2\ell+1,3\ell+1,\ldots,\\
            \vdots\\
            \gamma(t) &= v_\ell,\:\:t=\ell-1,\ell+(\ell-1),2\ell+(\ell-1),3\ell+(\ell-1),\ldots.
        \end{align*}
        Notice that \(\gamma\) is a well-defined scheduling sequence. Indeed, from \ref{mainprop1}, we have that each \(v_q\), \(q=1,2,\ldots,\ell\) contains \(M\) elements. It follows from \ref{mainprop2} that \(\gamma\) allows each plant \(i\in\{1,2,\ldots,N\}\) to access the shared communication network. In addition, by its construction, \(\gamma\) is periodic with period \(\ell\). We need to show that each plant \(i\in\{1,2,\ldots,N\}\) is ESMS under \(\gamma\).

        In view of Lemma \ref{lem:auxres2}, we have that for each plant \(i\in\{1,2,\ldots,N\}\), the transition probability matrix, \(\Pi_i\), is periodic with period \(\ell\). Fix \(i\in\D_q\), \(q\in\{1,2,\ldots,\ell\}\). By construction of \(\gamma\) and properties of \(\D_q\), the following holds: for some symmetric and positive definite matrices \(Q_{i,1}(j)\), \(Q_{i,2}(j),\ldots\), \(Q_{i,\ell}(j)\), \(j\in R_i\), there exist symmetric and positive definite matrices \(P_{i,1}(j)\), \(P_{i,2}(j),\ldots\), \(P_{i,\ell}(j)\), \(j\in R_i\), such that
        \begin{align}
        \label{e:pf1_step1}
            &\:(1-p)A_k^\top P_{i,\tau+1}(\is)A_k + pA_k^\top P_{i,\tau+1}(\iu)A_k - P_{i,\tau}(k)\nonumber\\
            =&\: p_{k\is}(\tau)A_k^\top P_{i,\tau+1}(\is)A_k + p_{k\iu}(\tau)A_k^\top P_{i,\tau+1}(\iu)A_k - P_{i,\tau}(k)\nonumber\\
            =&\: \sum_{j\in R_i}p_{kj}(\tau)A_k^\top P_{i,\tau+1}(j)A_k - P_{i,\tau}(k)\nonumber\\
            =&\: -Q_{i,\tau}(k),\:\:\text{if}\:i\in\D_{\tau+1},\:\tau=1,2,\ldots,\ell-1,\:k\in R_i,
        \end{align}
        \begin{align}
        \label{e:pf1_step2}
            &\:A_k^\top P_{i,\tau+1}(\iu)A_k - P_{i,\tau}(k)\nonumber\\
            =&\: 0\cdot A_k^\top P_{i,\tau+1}(\is) A_k + 1\cdot A_k^\top P_{i,\tau+1}(\iu)A_k - P_{i,\tau}(k)\nonumber\\
            =&\: p_{k\is}(\tau)A_k^\top P_{i,\tau+1}(\is)A_k + p_{k\iu}(\tau)A_k^\top P_{i,\tau+1}(\iu)A_k - P_{i,\tau}(k)\nonumber\\
            =&\: \sum_{j\in R_i}p_{kj}(\tau)A_k^\top P_{i,\tau+1}(j)A_k - P_{i,\tau}(k)\nonumber\\
            =&\: -Q_{i,\tau}(k),\:\:\text{if}\:i\notin\D_{\tau+1},\:\tau = 1,2,\ldots,\ell-1,\:k\in R_i,
        \end{align}
       \begin{align}
        \label{e:pf1_step3}
            &\:(1-p)A_k^\top P_{i,1}(\is)A_k + pA_k^\top P_{i,1}(\iu)A_k - P_{i,\ell}(k)\nonumber\\
            =&\: p_{k\is}(\ell)A_k^\top P_{i,1}(\is)A_k + p_{k\iu}(\ell)A_k^\top P_{i,1}(\iu)A_k - P_{i,\ell}(k)\nonumber\\
            =&\: \sum_{j\in R_i}p_{kj}(\ell)A_k^\top P_{i,1}(j)A_k - P_{i,\ell}(k)\nonumber\\
            =&\: -Q_{i,\ell}(k),\:\:\text{if}\:i\in\D_{1},\:k\in R_i,
        \end{align}
        and
        \begin{align}
        \label{e:pf1_step4}
            &\:A_k^\top P_{i,1}(\iu)A_k - P_{i,\ell}(k)\nonumber\\
            =&\: 0\cdot A_k^\top P_{i,1}(\is) A_k + 1\cdot A_k^\top P_{i,1}(\iu)A_k - P_{i,\ell}(k)\nonumber\\
            =&\: p_{k\is}(\ell)A_k^\top P_{i,1}(\is)A_k + p_{k\iu}(\ell)A_k^\top P_{i,1}(\iu)A_k - P_{i,\ell}(k)\nonumber\\
            =&\: \sum_{j\in R_i}p_{kj}(\ell)A_k^\top P_{i,1}(j)A_k - P_{i,\ell}(k)\nonumber\\
            =&\: -Q_{i,\ell}(k),\:\:\text{if}\:i\notin\D_{1},\:k\in R_i.
        \end{align}

        In view of Theorem \ref{t:recall_res}, periodicity of \(\Pi_i\) along with \eqref{e:pf1_step1}-\eqref{e:pf1_step4} imply ESMS of plant \(i\) under \(\gamma\). Since \(i\) and \(q\) were chosen arbitrarily, ESMS of all plants \(\displaystyle{i\in\bigcup_{q=1}^{\ell}\D_q}=\{1,2,\ldots,N\}\) follows.

        (Necessity) Consider a scheduling sequence, \(\gamma\), that is periodic with period \(\ell\) and ensures ESMS of all plants \(i\in\{1,2,\ldots,N\}\). We need to show that there exist distinct sets \(\D_q\subset\{1,2,\ldots,N\}\) that satisfy conditions \ref{mainprop1}-\ref{mainprop3}.

        Let us write \(\gamma\) as
        \begin{align*}
            \gamma(t) &= u_1,\:\:t=0,\ell,2\ell,3\ell,\ldots,\\
            \gamma(t) &= u_2,\:\:t=1,\ell+1,2\ell+1,3\ell+1,\ldots,\\
            \vdots\\
            \gamma(t) &= u_{\ell},\:\:t=\ell-1,\ell+(\ell-1),2\ell+(\ell-1),3\ell+(\ell-1),\ldots,
        \end{align*}
        where \(u_q\in\Svec\), \(q=1,2,\ldots,\ell\). Let \(\overline{\D}_q\) be the set that contains the elements of the vector \(u_q\), \(q=1,2,\ldots,\ell\). Clearly, \(\abs{\overline{\D}_q} = M\), \(q=1,2,\ldots,\ell\). Since \(\gamma\) ensures ESMS of each plant \(i\in\{1,2,\ldots,N\}\), it must be true that \(\gamma\) allows each plant \(i\) to access the shared communication network. Indeed, by Assumption \ref{a:stability} the open-loop dynamics of the plants are unstable. It follows that \(\displaystyle{\bigcup_{q=1}^{\ell}\overline{\D}_q = \{1,2,\ldots,N\}}\). In addition, by construction of \(\gamma\), it is periodic with period \(\ell\), thereby ensuring that the sets \(\overline{\D}_q\), \(q=1,2,\ldots,\ell\) are distinct.

        From Lemma \ref{lem:auxres2}, it follows that for each plant \(i\in\{1,2,\ldots,N\}\), the transition probability matrix, \(\Pi_i\), is periodic with period \(\ell\). Fix \(i\in\overline{\D}_q\), \(q\in\{1,2,\ldots,\ell\}\). In view of Theorem \ref{t:recall_res}, ESMS of plant \(i\) implies that for some symmetric and positive definite matrices \(Q_{i,1}(j)\), \(Q_{i,2}(j),\ldots\), \(Q_{i,\ell}(j)\), \(j\in R_i\), there exist symmetric and positive definite matrices \(P_{i,1}(j)\), \(P_{i,2}(j),\ldots\), \(P_{i,\ell}(j)\), \(j\in R_i\), such that conditions \eqref{e:maincondn1}-\eqref{e:maincondn2} hold. By construction of \(\gamma\), we have
        \begin{align}
        \label{e:pf1_step5}
            &\:\sum_{j\in R_i}p_{kj}(\tau)A_k^\top P_{i,\tau+1}(j)A_k - P_{i,\tau}(k)\nonumber\\
            =&\: p_{k\is}(\tau)A_k^\top P_{i,\tau+1}(\is)A_k + p_{k\iu}(\tau)A_{k}^\top P_{i,\tau+1}(\iu)A_k - P_{i,\tau}(k)\nonumber\\
            =&\: (1-p)A_k^\top P_{i,\tau+1}(\is)A_k + pA_k^\top P_{i,\tau+1}(\iu)A_k - P_{i,\tau}(k)\nonumber\\
            =&\: -Q_{i,\tau}(k),\:\:\text{if}\:i\in\D_{\tau+1},\:\tau=1,2,\ldots,\ell-1,\:k\in R_i,
        \end{align}
        \begin{align}
        \label{e:pf1_step6}
            &\:\sum_{j\in R_i}p_{kj}(\tau)A_k^\top P_{i,\tau+1}(j)A_k - P_{i,\tau}(k)\nonumber\\
            =&\: p_{k\is}(\tau)A_k^\top P_{i,\tau+1}(\is)A_k + p_{k\iu}(\tau)A_k^\top P_{i,\tau+1}(\iu)A_k - P_{i,\tau}(k)\nonumber\\
            =&\: 0\cdot A_k^\top P_{i,\tau+1}(\is)A_k + 1\cdot A_k^\top P_{i,\tau+1}(\iu)A_k - P_{i,\tau}(k)\nonumber\\
            =&\: A_k^\top P_{i,\tau+1}(\iu)A_k - P_{i,\tau}(k)\nonumber\\
            =&\: -Q_{i,\tau}(k),\:\:\text{if}\:i\notin\D_{\tau+1},\:\tau = 1,2,\ldots,\ell-1,\:k\in R_i,
        \end{align}
        \begin{align}
        \label{e:pf1_step7}
            &\:\sum_{j\in R_i}p_{kj}(\ell)A_k^\top P_{i,1}(j)A_k - P_{i,\ell}(k)\nonumber\\
            =&\: p_{k\is}(\ell)A_k^\top P_{i,1}(\is)A_k + p_{k\iu}(\ell)A_{k}^\top P_{i,1}(\iu)A_k - P_{i,\ell}(k)\nonumber\\
            =&\: (1-p)A_k^\top P_{i,1}(\is)A_k + pA_k^\top P_{i,1}(\iu)A_k - P_{i,\ell}(k)\nonumber\\
            =&\: -Q_{i,\ell}(k),\:\:\text{if}\:i\in\D_{1},\:k\in R_i,
        \end{align}
        \begin{align}
        \label{e:pf1_step8}
            &\:\sum_{j\in R_i}p_{kj}(\ell)A_k^\top P_{i,1}(j)A_k - P_{i,\ell}(k)\nonumber\\
            =&\: p_{k\is}(\ell)A_k^\top P_{i,1}(\is)A_k + p_{k\iu}(\ell)A_k^\top P_{i,1}(\iu)A_k - P_{i,\ell}(k)\nonumber\\
            =&\: 0\cdot A_k^\top P_{i,1}(\is)A_k + 1\cdot A_k^\top P_{i,1}(\iu)A_k - P_{i,\ell}(k)\nonumber\\
            =&\: A_k^\top P_{i,1}(\iu)A_k - P_{i,\ell}(k)\nonumber\\
            =&\: -Q_{i,\ell}(k),\:\:\text{if}\:i\notin\D_{1},\:k\in R_i.
        \end{align}
        Since \(i\) and \(q\) were chosen arbitrarily, it follows from \eqref{e:pf1_step5}-\eqref{e:pf1_step8} that each element \(i\in\overline{\D}_q\), \(q=1,2,\ldots,\ell\), satisfies conditions \eqref{e:condn1}-\eqref{e:condn4}. We conclude that there exist \(\D_q = \overline{\D}_q\subset\{1,2,\ldots,N\}\), \(q=1,2,\ldots,\ell\), that satisfy conditions \ref{mainprop1}-\ref{mainprop3}.

        This completes our proof of Theorem \ref{t:mainres}.
    \end{proof}

\begin{proof}[Proof of Proposition \ref{prop:algores}]
Our proof of Proposition \ref{prop:algores} will rely on the following:
	\begin{lemma}
	\label{lem:auxres3}
		Consider \((\D_q)_{q=1}^{\ell}\) that satisfies conditions \ref{mainprop1}-\ref{mainprop2}. The following are equivalent:
		\begin{enumerate}[label = \roman*), leftmargin = *]
			\item Each element \(i\in\D_q\), \(q=1,2,\ldots,\ell\), satisfies that for some symmetric and positive definite matrices \(Q_{i,1}(j)\), \(Q_{i,2}(j),\ldots\), \(Q_{i,\ell}(j)\), \(j\in R_i\), there exist symmetric and positive definite matrices \(P_{i,1}(j)\), \(P_{i,2}(j),\ldots\), \(P_{i,\ell}(j)\), \(j\in R_i\), such that conditions \eqref{e:condn1}-\eqref{e:condn4} hold.
			\item The feasibility problem \eqref{e:feasprob} admits a solution to each \(i\in\D_q\), \(q=1,2,\ldots,\ell\).
		\end{enumerate}
	\end{lemma}
	
	\begin{proof}
		i)\(\implies\)ii): Since the matrices \(Q_{i,1}(j)\), \(Q_{i,2}(j),\ldots\), \(Q_{i,\ell}(j)\), \(j\in R_i\), are symmetric and positive definite, the expressions on the left-hand side of the equalities \eqref{e:condn1}-\eqref{e:condn4} must be symmetric and negative definite matrices.\\
		ii)\(\implies\)i):  We have that the expressions on the left-hand side of the equalities \eqref{e:condn1}-\eqref{e:condn4} are symmetric and negative definite matrices. Then their negations are symmetric and positive definite matrices.
		
		The assertion of Lemma \ref{lem:auxres3} follows at once.
	\end{proof}

    \begin{proof}[Proof of Proposition \ref{prop:algores}]
    	\begin{enumerate}[label = \roman*), leftmargin =*]
    		\item  In view of Definition \ref{d:periodic-sequence}, a \(\gamma\) obtained from Algorithm \ref{algo:sched-policy_design} is, by construction, periodic with period \(\ell\). It remains to show that \(\gamma\) is stabilizing.
	
	By Lemma \ref{lem:auxres3}, the choice of \((\D_q)_{q=1}^{\ell}\) employed to design \(\gamma\) in Algorithm \ref{algo:sched-policy_design} satisfies conditions \ref{mainprop1}-\ref{mainprop3}. From Theorem \ref{t:mainres}, it follows that the existence of \((\D_q)_{q=1}^{\ell}\) that satisfies conditions \ref{mainprop1}-\ref{mainprop3} implies the existence of a periodic scheduling sequence with period \(\ell\) that ensures ESMS of plant \(i\in\{1,2,\ldots,N\}\). The fact that \(\gamma\) constructed in Algorithm \ref{algo:sched-policy_design} is one such sequence, is shown mathematically in our proof of Theorem \ref{t:mainres} (sufficiency part).
	
	\item A failure message obtained from Algorithm \ref{algo:sched-policy_design} implies that for every \((\D_q)_{q=1}^{\ell}\) that satisfies conditions \ref{mainprop1}-\ref{mainprop2}, there exists at least one \(i\in\D_q\), \(q\in\{1,2,\ldots,\ell\}\) such that the feasibility problem \eqref{e:feasprob} does not admit a solution. In view of Lemma \ref{lem:auxres3}, we have that the above is equivalent to the non-existence of \((\D_q)_{q=1}^{\ell}\) that satisfies \ref{mainprop1}-\ref{mainprop3}. From Theorem \ref{t:mainres}, it follows that the NCS under consideration does not admit a periodic scheduling sequence with period \(\ell\) that ensures ESMS of each plant \(i\in\{1,2,\ldots,N\}\).
	\end{enumerate}
	This completes our proof of Proposition \ref{prop:algores}.
    \end{proof}
\end{proof}

 {
    \begin{proof}[Proof of Proposition \ref{prop:mainres}]
    	We need to show that if there exists \((\D_q)_{q=1}^{\ell}\in\Gamma\) such that the feasibility problems \eqref{e:feasprob1}-\eqref{e:feasprob2} admit solutions for each \(i\in\{1,2,\ldots,N\}\), then there exist symmetric and positive definite matrices \(P_{i,1}(j),\ldots\), \(P_{i,\ell}(j)\), \(j\in R_i\) such that the following conditions hold:
	\begin{align}
		\label{e:pf3_step1}(1-p)A_{\iu}^\top P_{i,\tau+1}(\is)A_{\iu} + pA_{\iu}^\top P_{i,\tau+1}(\iu)A_{\iu} - P_{i,\tau}(\iu) &\prec 0,\:\:\text{if}\:i\in\D_{\tau+1},\:\tau=1,2,\ldots,\ell-1,\\
		\label{e:pf3_step2}(1-p)A_{\is}^\top P_{i,\tau+1}(\is)A_{\is} + pA_{\is}^\top P_{i,\tau+1}(\iu)A_{\is} - P_{i,\tau}(\is) &\prec 0,\:\:\text{if}\:i\in\D_{\tau+1},\:\tau=1,2,\ldots,\ell-1,\\
                 \label{e:pf3_step3}A_{\iu}^\top P_{i,\tau+1}(\iu)A_{\iu} - P_{i,\tau}(\iu) &\prec 0,\:\:\text{if}\:i\notin\D_{\tau+1},\:\tau = 1,2,\ldots,\ell-1,\\
                 \label{e:pf3_step4}A_{\is}^\top P_{i,\tau+1}(\iu)A_{\is} - P_{i,\tau}(\is) &\prec 0,\:\:\text{if}\:i\notin\D_{\tau+1},\:\tau = 1,2,\ldots,\ell-1,\\
                 \label{e:pf3_step5}(1-p)A_{\iu}^\top P_{i,1}(\is)A_{\iu} + pA_{\iu}^\top P_{i,1}(\iu)A_{\iu} - P_{i,\ell}(\iu) &\prec 0,\:\:\text{if}\:i\in\D_1,\\
                 \label{e:pf3_step6}(1-p)A_{\is}^\top P_{i,1}(\is)A_{\is} + pA_{\is}^\top P_{i,1}(\iu)A_{\is} - P_{i,\ell}(\is) &\prec 0,\:\:\text{if}\:i\in\D_1,\\
                 \label{e:pf3_step7}A_{\iu}^\top P_{i,1}(\iu)A_{\iu} - P_{i,\ell}(\iu) &\prec 0,\:\:\text{if}\:i\notin\D_1,\\
                \label{e:pf3_step8}A_{\is}^\top P_{i,1}(\iu)A_{\is} - P_{i,\ell}(\is) &\prec 0,\:\:\text{if}\:i\notin\D_1.
	\end{align}
	
	From \eqref{e:feasprob1} it follows that there exist symmetric and positive definite matrices \(P_{i,1}(\iu),\ldots\), \(P_{i,\ell}(\iu)\) and \(P_{i,q}(\is)\), where \(i\in\D_q\), \(q\in\{1,2,\ldots,\ell\}\) such that \eqref{e:pf3_step1}, \eqref{e:pf3_step3}, \eqref{e:pf3_step5} and \eqref{e:pf3_step7} are satisfied. Now, consider that there exist \(q\in\{1,2,\ldots,\ell\}\) and \(Y_i\) such that \eqref{e:feasprob2} admits a solution. By Schur complement \cite[Definition 6.1.8]{Bernstein}, the inequalities in \eqref{e:feasprob2} are equivalent to
	\[
		\pmat{-P_{i,q+1}^{-1}(\iu) & A_i P_{i,q}^{-1}(\is)+B_i Y_i\\\star & -P_{i,q}^{-1}(\is)}\prec 0
	\]
	 \[
		\text{(resp.,}\:\:\pmat{-\bigl((1-p)P_{i,q+1}(\is)+p P_{i,q+1}(\iu)\bigr)^{-1} & A_i P_{i,q}^{-1}(\is)+B_i Y_i\\\star & -P_{i,q}^{-1}(\is)}\prec 0\:\:).
	\]
	The above expression can be written as
	\[
		\diag\bigl(P_{i,q+1}^{-1}(\iu),P_{i,q}^{-1}(\is)\bigr)^\top\pmat{-P_{i,q+1}(\iu) & P_{i,q+1}(\iu)(A_i+B_i K_i)\\\star & -P_{i,q}(\is)}\diag\bigl(P_{i,q+1}^{-1}(\iu),P_{i,q}^{-1}(\is)\bigr)\prec 0
	\]
	\begin{align*}
		\text{(resp.,}\:\:&\diag\biggl(\bigl((1-p)P_{i,q+1}(\is)+pP_{i,q+1}(\iu)\bigr)^{-1},P_{i,q}^{-1}(\is)\biggr)^\top\times\\
		&\pmat{-\bigl((1-p)P_{i,q+1}(\is)+pP_{i,q+1}(\iu)\bigr) & \bigl((1-p)P_{i,q+1}(\is)+pP_{i,q+1}(\iu)\bigr)(A_i+B_i K_i)\\\star & -P_{i,q}(\is)}\times\\
		&\diag\biggl(\bigl((1-p)P_{i,q+1}(\is)+pP_{i,q+1}(\iu)\bigr)^{-1},P_{i,q}^{-1}(\is)\biggr)\prec 0\:\:),
	\end{align*}
	where \(K_i\) is described in Step 9. of Algorithm \ref{algo:contrl_design}. Since the left-hand side of the above inequality is a congruent transformation \cite[Definition 3.4.4]{Bernstein} of \(\pmat{-P_{i,q+1}(\iu) & P_{i,q+1}(\iu)(A_i+B_i K_i)\\\star & -P_{i,q}(\is)}\) (resp., \\\(\pmat{-\bigl((1-p)P_{i,q+1}(\is)+pP_{i,q+1}(\iu)\bigr) & \bigl((1-p)P_{i,q+1}(\is)+pP_{i,q+1}(\iu)\bigr)(A_i+B_i K_i)\\\star & -P_{i,q}(\is)}\)), it holds that\\
	\(\pmat{-P_{i,q+1}(\iu) & P_{i,q+1}(\iu)(A_i+B_i K_i)\\\star & -P_{i,q}(\is)}\prec 0\) (resp.,\\ \(\pmat{-\bigl((1-p)P_{i,q+1}(\is)+pP_{i,q+1}(\iu)\bigr) & \bigl((1-p)P_{i,q+1}(\is)+pP_{i,q+1}(\iu)\bigr)(A_i+B_i K_i)\\\star & -P_{i,q}(\is)}\prec 0\)). By Schur complement, the above inequality is the same as
	\[
		(A_i+B_i K_i)^\top P_{i,q+1}(\iu)(A_i+B_iK_i)-P_{i,q}(\is)\prec 0
	\]
	\[
		\text{(resp.,}\:\:(A_i+B_i K_i)^\top \bigl((1-p)P_{i,q+1}(\is)+pP_{i,q+1}(\iu)\bigr)(A_i+B_iK_i)-P_{i,q}(\is)\prec 0\:\:).
	\]
	From \eqref{e:feasprob3} we have that there exist symmetric and positive definite matrices \(P_{i,1}(\is),\ldots\), \(P_{i,q-1}(\is)\), \(P_{i,q+1}(\is),\ldots\), \(P_{i,\ell}(\is)\) which together with \(P_{i,q}(\is)\), \(i\in\D_q\), \(q\in\{1,2,\ldots,\ell\}\) obtained above satisfy \eqref{e:pf3_step2}, \eqref{e:pf3_step4}, \eqref{e:pf3_step6} and \eqref{e:pf3_step8}.
	
	We conclude that the feasibility problem \eqref{e:feasprob} admits a solution \(P_{i,1}(j),\ldots\), \(P_{i,\ell}(j)\), \(j\in R_i\), \(i=1,2,\ldots,N\). This completes our proof of Proposition \ref{prop:mainres}.
    \end{proof}
    }



\end{document}